\documentclass[11pt, notitlepage]{amsart}
\usepackage{latexsym,amsfonts,amssymb,amsmath,amsthm}
\usepackage{graphicx}
\usepackage{color}
\usepackage{mathrsfs}
\usepackage{tabularx}
\usepackage{wrapfig}
\usepackage{stackrel}
\usepackage{color}
\usepackage{float}
\usepackage{enumitem}
\usepackage[hidelinks]{hyperref}
\setlist[enumerate]{leftmargin=1.5em,itemsep=2pt}
\setlist[itemize]{leftmargin=1.2em,itemsep=2pt}
\usepackage{mathtools}
\usepackage{multicol}
\usepackage[normalem]{ulem}
\usepackage{cancel}
\usepackage{graphicx}
\usepackage{multirow}
\graphicspath{ {./images/} }
\usepackage{stmaryrd}
\usepackage{xparse}
\usepackage{csquotes}
\usepackage[UKenglish]{babel}
\usepackage[hang]{footmisc}
\usepackage{subfiles}
\usepackage{tikz, tikz-cd, float} 
\usetikzlibrary{positioning}
\numberwithin{equation}{section}

\usepackage[margin=1in]{geometry}
\usepackage[backend=biber, isbn=false, doi=false, giveninits=true, style=alphabetic, useprefix=true, maxcitenames=4, maxbibnames=4, sorting=nyt,citestyle=alphabetic]{biblatex}

\DeclareFieldFormat[misc]{title}{\mkbibemph{#1}}
\DeclareFieldFormat[article]{title}{\mkbibemph{#1}}
\DeclareFieldFormat[incollection]{title}{\mkbibemph{#1}}

\newtheorem{thm}{Theorem}[section]
\newtheorem{prop}[thm]{Proposition}
\newtheorem{lemma}[thm]{Lemma}
\newtheorem{cor}[thm]{Corollary}

\theoremstyle{definition}
\newtheorem{defn}[thm]{Definition}
\newtheorem{rem}[thm]{Remark}

\newtheorem*{prop*}{Proposition}

\newcommand{\Alg}{\mathrm{Alg}}
\newcommand{\C}{\mathbb{C}}
\newcommand{\F}{\mathbb{F}}
\newcommand{\Q}{\mathbb{Q}}
\newcommand{\R}{\mathbb{R}}
\newcommand{\Z}{\mathbb{Z}}

\newcommand{\ev}{\mathrm{ev}}

\newcommand{\End}{\mathrm{End}}
\newcommand{\Iw}{\mathrm{Iw}}

\newcommand{\GL}{\mathrm{GL}}
\newcommand{\Gal}{\mathrm{Gal}}
\newcommand{\GSp}{\mathrm{GSp}}
\newcommand{\Hom}{\mathrm{Hom}}
\newcommand{\PGSp}{\mathrm{PGSp}}
\newcommand{\GSpin}{\mathrm{GSpin}}

\newcommand{\PGL}{\mathrm{PGL}}
\newcommand{\A}{\mathbb{A}}

\newcommand{\Art}{\mathrm{Art}}

\newcommand{\Frob}{\mathrm{Frob}}
\newcommand{\WD}{\mathrm{WD}}

\newcommand{\Ext}{\mathrm{Ext}}
\newcommand{\Dcris}{D_{\mathrm{cris}}}

\usepackage{geometry}

\title{$p$-adic rigidity for $\GSp_4$}
\author{Charlotte Clare-Hunt} 
\address{Mathematical Institute, University of Oxford, Woodstock Road, Oxford, OX2 6GG, UK}
\email{charlotte.clare-hunt@maths.ox.ac.uk}

\makeatletter
\patchcmd{\@setaddresses}{\nobreak \indent}{\nobreak \noindent}{}{}
\patchcmd{\@setaddresses}{\par \addvspace \bigskipamount \indent}{\par \addvspace \bigskipamount \noindent}{}{}
\patchcmd{\@setaddresses}{\indent \emailaddrname}{\noindent \emailaddrname}{}{}
\patchcmd{\@setaddresses}{\indent \urladdrname}{\noindent \urladdrname}{}{}
\makeatother

\begin{document}
\begin{abstract}
This paper establishes the $p$-adic rigidity of certain suitably refined noncuspidal automorphic Saito--Kurokawa representations of $\GSp_4(\mathbb{A}_\mathbb{Q})$, in the sense that they cannot be interpolated in a nontrivial positive dimensional $p$-adic family. The results provide a $\GSp_4$ analogue of Bella\"iche’s rigidity theorems for $\mathrm{U}(2,1)$ and identify an obstruction to $p$-adic variation on the $\GSp_4$ eigenvariety.
\end{abstract}

\maketitle

\tableofcontents
\section{Introduction}
\subsection{Eigenvarieties and $p$-adic families}

There has been substantial progress in the construction of eigenvarieties, $p$-adic rigid analytic spaces that interpolate systems of Hecke eigenvalues of certain algebraic automorphic representations, and the study of $p$-adic families of automorphic forms--see \cite{Eme11} for a survey. One of the main motivations for such constructions involves the relationship between $p$-adic families, Selmer groups and ($p$-adic) $L$-functions. In particular, families of Eisenstein series have played an important role in arithmetic applications, from Ribet's converse to Herbrand's theorem to Iwasawa theory \cite{MW84}. The geometry of the noncuspidal, nonholomorphic locus remains lesser-understood, motivating the study of the noncuspidal Saito--Kurokawa lifts of Schmidt \cite{Sch05}, which are noncuspidal only in the nonholomorphic case. 

\subsection{$p$-adic rigidity}
Antipodal to the construction of $p$-adic deformations is complete $p$-adic rigidity---the nonexistence of nontrivial positive-dimensional $p$-adic families through a given point. To prove rigidity for the $p$-refined representations that we study, in a framework that does not rely on the existence of an eigenvariety, we abstract the properties that a $p$-adic family for the symplectic group $\GSp_4$ is expected to satisfy, adapting Bella\"iche's definition for the quasisplit unitary group $\mathrm{U}(2,1)$ in \cite{Bel10}.

Let $\Pi$ be an automorphic representation of $\GSp_4(\A_\Q)$, unramified at a rational prime $p$, and of \emph{Galois type} (Definition~\ref{def: aut rep of Galois type}), and let $T$ be the torus of the upper triangular Borel of $\GSp_4$. Attached to a $p$-adic family should be a set of classical points, each consisting of a pair $(\Pi, \psi_{\mathcal{R}})$, where $\psi_\mathcal{R}$ is an \emph{accessible refinement} of $\Pi_p$; a character $\psi_\mathcal{R}:T(\mathbb{Q}_p) \to \C^\times$ occurring in the normalised Jacquet module of $\Pi_p$ and which is trivial on $T(\Z_p)$. Such refinements may be identified with the crystalline Frobenius eigenvalues of the associated $p$-adic Galois representation $\rho_{\Pi}$, cf. Definition~\ref{def: refinement-local} and Proposition~\ref{prop:Fi}. An unramified Saito--Kurokawa local factor usually has four distinct accessible refinements $\psi_1, \ldots, \psi_4$ (cf. Remark~\ref{rem: eval conj}). We work with $p$-adic families with a Zariski dense, self-accumulating subset of classical points whose $\GSp_4(\A_\Q)$-representations have fixed central character, which we may (without loss of generality) take to be trivial.

\subsection{Saito--Kurokawa points}
We study $p$-adic families through a special class of \emph{Saito--Kurokawa (SK) points}, namely those for which the automorphic $\GSp_4(\A_\Q)$-representation $\pi$ is the lift of a $\PGSp_4(\A_\Q)$ Saito--Kurokawa representation of Schmidt \cite{Sch05}. Attached to $\pi$ is a $\GL_2(\A_\Q)$-representation $\mu$ that corresponds to a weight $2k(z_0)-2$ holomorphic newform $f_{z_0}$ with $k(z_0) \in \Z_{\ge 2}$ and $p \nmid \mathrm{cond}(\mu)$. Arthur’s classification \cite{Art04} of discrete automorphic representations of $\GSp_4$ plays, for these points, the same structural role that Rogawski’s classification \cite{Rog90} plays in Bella\"iche’s rigidity for $\mathrm{U}(2,1)$ \cite{Bel10}. An auxiliary `SK rigidity' result for families with a Zariski dense set of SK points, analogous to Bella\"iche’s rigidity in cohomological degree one, highlights the significance of the refinement for the rigidity of SK points (we expect that $\psi_1$- and $\psi_4$-refined points do deform in such families, cf.~Remark \ref{rem: semiordinary family}). This is one reason why our main theorem applies only to $\psi_2$- and $\psi_3$-refined SK points $z_0$, whose refinements are analogous to Bella\"iche's antiordinary ones. For $\psi_3$ we further assume $v(\alpha(z_0))\neq k(z_0)-2$, where $v(\alpha(z_0))$ is the $U_p$-slope of $f_{z_0}$, cf. Proposition~\ref{prop:ref 2/3}. For complete rigidity we further assume that $z_0$ is noncuspidal and satisfies 
\begin{enumerate}[leftmargin = 2em]
    \item[(St)] for every prime $\ell\neq p$ such that $\mu_{\ell}$ is an unramified twist of Steinberg, $\mu_\ell \simeq \mathrm{St} \otimes \xi$, where $\xi$ is the unramified character with $\xi(\ell) = -1$.
\end{enumerate}
We require two further inputs. The first is the \emph{Kisin property} at $z_0$, an interpolation property for crystalline Frobenius eigenvalues that is expected to hold for eigenvarieties and, more generally, for $p$-adic families with a sufficiently dense set of classical points. The second is a monodromy control condition \textup{(SK--P2)}, which encodes the family‑level Galois monodromy away from $p$. 
\begin{thm}[Theorem~\ref{thm: X is a point}]
\label{thm: complete rigidity}
Let $z_0$ be a noncuspidal $\psi_i$-refined SK point satisfying \textup{(St)} with $i\in\{2,3\}$ and $v(\alpha(z_0))\neq k(z_0)-2$ when $i=3$. Let $X(1)$ be an irreducible $p$-adic family interpolating automorphic representations of $\GSp_4(\A_\Q)$ of Galois type and level $K$ with trivial central character through $z_0$ satisfying condition \textup{(SK--P2)} and the Kisin property at $z_0$. Then $X(1)$ is a point.
\end{thm}
\noindent
\begin{rem}\label{rem: noncuspidal motivation short}
A motivation for the noncuspidality of $z_0$ is that cuspidal points are never rigid \cite[Theorem 1.4.3]{BP21}: indeed, holomorphic SK points do deform in positive-dimensional $p$-adic families \cite{BB22, SU06}.
\end{rem}

\subsection{Method of proof}
The strategy is as follows. Suppose that $\dim X > 0$. Via the SK rigidity (\S \ref{s: Rigidity for families of SK points}) and reducibility results on pseudocharacters (\S \ref{s:Pseudocharacters and the form of $T$}), we show that the generic pseudocharacter $T_\eta: G_{\Q, S} \to \mathrm{Frac}(\mathcal{O}(X))$ attached to $X(1)$ is not $(2,1,1)$-reducible (Proposition~\ref{prop: T-not-(2,1,1)}). The GMA machinery of \cite{BC09} gives a nonsplit extension of Galois representations in one of two groups of extensions $\mathrm{Ext}_T(\epsilon^{-2},\epsilon^{-1})$ or $\mathrm{Ext}_T(\epsilon^{-2},\rho_\mu)$ attached to the pseudocharacter $T:G_{\Q, S} \to \mathcal{O}_{z_0}$. Such an extension is crystalline at $p$ (by the Kisin property) and unramified outside $p$ (by \textup{(SK--P2)}, nongenericity of $z_0$, and \textup{(St)}), so gives rise to a nontrivial class in one of two Bloch--Kato Selmer groups $H_f^1(\Q,\rho_\mu(2))$ or $H_f^1(\Q, \overline{\Q_p}(1))$ that vanish (the former since $L(1/2, \mu) \neq 0$, the latter by Kummer theory), a contradiction. 

\subsection{Structure of the paper}
In \S \ref{s: Preliminaries} we set up some preliminaries for representations of $\GSp_4$. In \S \ref{s: Local representation theory} we compute the accessible refinements for (unramified at $p$) Saito--Kurokawa representations. In \S \ref{s: The Hecke algebra and Hecke operators} we recall the relationship between refinements and Hecke eigenvalues. In \S \ref{s: family} we define $p$-adic families. In \S \ref{s: Rigidity for families of SK points} we prove the auxiliary Saito--Kurokawa rigidity results.
In \S \ref{s:Pseudocharacters and the form of $T$} we analyse the pseudocharacter attached to a $p$-adic family and its reducibility. Finally in \S \ref{s:tot-rig} we prove Theorem~\ref{thm: complete rigidity}.

\subsection{Acknowledgments}
I would like to thank my supervisor James Newton for the inspiration for this problem and many invaluable discussions. I am also grateful to Andrew Graham and Hanneke Wiersema for their helpful comments and suggestions, and to Alex Horawa and Zach Feng for numerous useful conversations. This work was supported by  an EPSRC postgraduate studentship.
\section{Preliminaries} \label{s: Preliminaries}

\subsection{Lie theory for $\GSp_4$} \label{s: lie theory}
Following \cite{Box+21}, we define $\GSp_4$ to be the reductive group over $\mathbb{Z}$ defined as a subgroup of $\GL_4$, whose $R$-points are 
\begin{equation*}
    \GSp_4
(R) = \{g \in \GL_4(R) : gJg^t = \nu(g)J\},
\end{equation*}
where $R$ is a commutative ring, the map $\nu : \GSp_4 \to \mathbb{G}_m, \  g \mapsto \nu(g)$ is the similitude character, and $J$ is the antisymmetric matrix
\begin{equation*}
    J= \begin{pmatrix}
        0 & s \\ -s & 0
    \end{pmatrix}, \qquad \ s = \begin{pmatrix}
    0 & 1 \\ 1 & 0
\end{pmatrix}. 
\end{equation*}
Denote by $B \subset G = \GSp_4$ the Borel subgroup of upper triangular matrices, $T \subset B$ the diagonal maximal torus, and $Z_G=\{\mathrm{diag}(z, z, z, z) : z \in \mathbb{G}_m\}$ the centre of $G$. Write $X ^{*}(T)$ and $X_{*}(T)$ for the group of characters and cocharacters of $T$ respectively. We identify $X^{*}(T)$ with the set of triples $(a, b; c) \in \mathbb{Z}^3$ such that $c \equiv a + b \pmod 2$ via 
\begin{equation*}
    \lambda: t=\mathrm{diag}(t_1, t_2, \nu t_2^{-1}, \nu t_1^{-1}) \mapsto t_1^a t_2^b \nu^{(c-a-b)/2}.
\end{equation*}
In particular, the central character is given by $\lambda(\mathrm{diag}(z, z, z, z)) = z^c$ and the similitude character restricts to $(0, 0; 2)$ on $T$. The simple roots are $\alpha_1 = (1, -1; 0)$ and $\alpha_2 = (0, 2; 0)$, with $\alpha_1$ short and $\alpha_2$ long. The Weyl group $W_G = N_G(T)/T$ of $(G, T)$ acts on $X^*(T)$ by left conjugation $(\lambda\cdot w)(t) = \lambda(wtw^{-1})$. It is generated by the reflections
\begin{equation*}
    s_1=\begin{pmatrix}
        s & 0_2 \\ 0_2 & s
    \end{pmatrix}, \qquad
    s_2 = \begin{pmatrix}
        1 & & \\
        & s' & \\
        & & 1
    \end{pmatrix}, \qquad
    s' = \begin{pmatrix}
        0 & 1 \\ -1 & 0
    \end{pmatrix},
\end{equation*}
and admits the presentation $W_G = \langle s_1, s_2 \mid s_1^2=s_2^2=(s_1s_2)^4 = 1 \rangle$.
The root datum $(G,B,T)$ determines the dual root datum $(\widehat{G},\widehat{B},\widehat{T})$, where the Langlands dual group $\widehat{G}=\GSpin_5$ is identified with $\GSp_4$ via the spin isomorphism (cf.\ \cite[\S3.2]{MT02}). The cocharacter in $X_*(\widehat{T})$ corresponding to $(a,b;c)\in X^*(T)$ is
\begin{equation*}
    t \mapsto \mathrm{diag}(t^{(a+b+c)/2}, t^{(a-b+c)/2}, t^{(-a+b+c)/2}, t^{(-a-b+c)/2}).
\end{equation*}
If $t= \mathrm{diag}(t_1, t_2, \nu t_2^{-1}, \nu t_1^{-1})=:[t_1,t_2; \nu] \in T$, then $\alpha_1(t)=t_1t_2^{-1}$ and $\alpha_2(t)=t_2^2\nu^{-1}$. If $v$ is a finite place of a number field $F$ with ring of integers $\mathcal{O}_{F_v}$ and residue field $k(v)$, then we have the standard subgroups of $\GSp_4(F_v)$:
\begin{itemize}[leftmargin=1.1em]
    \item The hyperspecial subgroup $\GSp_4(\mathcal{O}_{F_v})$.
    \item The Iwahori subgroup $\mathrm{Iw}(v)$, the preimage of $B(k(v))$ in $\GSp_4(\mathcal{O}_{F_v})$.
\end{itemize}
We denote by $P$ the Siegel parabolic of $\GSp_4$, whose Levi is given by
\begin{equation*}
    L(P)
  =\left\{
      \begin{pmatrix} A & \\ & uA' \end{pmatrix}
      : A\in\GL_2,\ u\in\GL_1
    \right\}
  \simeq \GL_2\times\GL_1,
  \qquad
  A' = s\, {}^t\!A^{-1} s.
\end{equation*}
The positive roots of $T$ (with respect to $B$) are $R^+(T)=\{\alpha_1, \alpha_2, \alpha_1+\alpha_2, 2\alpha_1+\alpha_2\}$. The modulus character $\delta_B:B(F_v) \to \mathbb{R}_{\geq 0}$, \ $\delta_B(x) = \lvert \det{(\mathrm{Ad}(x)) \vert_{\mathrm{Lie}(N_B)}}\rvert_v$, for $N_B$ the unipotent radical of $B$ and $\lvert \cdot \rvert_v$ the normalised\footnote{so that $\lvert \varpi_v \rvert_v = q_v^{-1}$ for any uniformiser $\varpi_v$ of $F_v$} absolute value on $F_v$ is given by
\begin{equation*}
    \delta_B\begin{psmallmatrix}
        t_1 & * & * & * \\
        & t_2 & * & * \\
        & & \nu t_2^{-1} &* \\
        & & & \nu t_1^{-1}
    \end{psmallmatrix}= \delta_B([t_1,t_2; \nu])=|(4\alpha_1+3\alpha_2)(t)|_v = \frac{|t_1|_v^4 |t_2|_v^2}{|\nu|_v^3}.
\end{equation*}

\subsection{The non-archimedean Local Langlands correspondence} \label{s: The non-archimedean Local Langlands correspondence} 
Let $\ell$ be a rational prime and $K/\mathbb{Q}_\ell$ be finite. We denote by $\mathrm{rec}_K$ and $\mathrm{rec}_{\mathrm{GT}}$ the local Langlands correspondences for $\GL_n(K)$ and $\GSp_4(K)$ defined in \cite{HT01} and \cite{GT11a} respectively, so that 
\begin{enumerate}[leftmargin=1.5em,itemsep=2pt]
    \item[$(i)$] if $\pi$ is an irreducible, complex, admissible representation of $\GL_n(K)$, then $\mathrm{rec}_K(\pi)$ is a Frobenius– semisimple Weil--Deligne representation of $W_K$ with coefficients in $\GL_n(\C)$.
    \item[$(ii)$] $\mathrm{rec}_\mathrm{GT}$ is a surjective, finite-to-one map from the set of equivalence classes of irreducible smooth complex representations of $\GSp_4(K)$ to the set of $\GSp_4$-conjugacy classes of $\GSp_4(\mathbb{C})$-valued Weil--Deligne representations of $W_K$, normalized so that if $\psi_\pi$ is the central character of $\pi$, then $\nu \circ \mathrm{rec}_\mathrm{GT}(\pi) = \mathrm{rec}_K(\psi_\pi)$ and $\mathrm{rec}_\mathrm{GT}(\pi \otimes (\chi \circ \nu)) = \mathrm{rec}_\mathrm{GT}(\pi) \otimes \mathrm{rec}_K(\chi)$.
\end{enumerate}
By \cite{CG15}, $\mathrm{rec}_{\mathrm{GT}}$ agrees with Arthur’s local Langlands correspondence for $\GSp_4$ and is compatible with the trace formula lifting from $\GSp_4$ to $\GL_4$. Fix an algebraic closure $\overline{\Q}/\Q$. For each prime $p$, fix once and for all embeddings $\overline{\Q} \hookrightarrow \overline{\Q}_p$ and $\overline{\Q} \hookrightarrow \C$, and an isomorphism $\iota_p : \overline{\Q}_p \xrightarrow{\sim} \C$ compatible with these embeddings. Write $\mathrm{rec}_{K,p} := \iota_p^{-1}\circ\mathrm{rec}_K\circ\iota_p,$ and $\mathrm{rec}_{\mathrm{GT},p} := \iota_p^{-1}\circ\mathrm{rec}_{\mathrm{GT}}\circ\iota_p,$ for the $\overline{\Q}_p$-valued versions of the above correspondences. The choice of $\iota_p$ does not play a role in our arguments, and we identify complex automorphic representations with their $\overline{\Q}_p$ ones via $\iota_p^{-1}$. In particular, $\iota_p$ fixes a choice of square root of $p$ in $\overline{\Q}_p$ corresponding to the positive square root in $\C$.

We use the cohomological normalisation of class field theory: for a non-archimedean local field $F$ of characteristic zero, the local Artin map $  \Art_F : F^\times \rightarrow W_F^{\mathrm{ab}}$ sends uniformisers to geometric Frobenii. For a number field $F$, the global Artin map $\Art_F : \A_F^\times \to G_F^{\mathrm{ab}}$ is the product of the local maps. If $F$ is a number field and $\phi:F^\times\backslash\A_F^\times\to\C^\times$ an algebraic Hecke character, denote by $\phi_p$ the associated global $p$-adic Galois character obtained via the global Artin map. For $v$ a finite place of $F$, let $\eta_v$ denote the id\`ele with component $\varpi_v$ at $v$ and $1$ elsewhere.

Let $F$ be a non-archimedean local field of characteristic zero. If $\rho:G_F \to \GL_n(\overline{\mathbb{Q}}_p)$ is a continuous representation, then we denote by $\WD(\rho)=(r, N)$ the associated Weil--Deligne representation, where $r:W_F\to\GL_n(\overline{\Q}_p)$, $N \in M_n(\overline{\Q}_p)$ and by $\WD(\rho)^{F-ss}$ its associated Frobenius semisimplification. If $\rho:G_F \to \GSpin_5(\overline{\mathbb{Q}}_p) \simeq \GSp_4(\overline{\mathbb{Q}}_p)$ is a continuous representation, then we denote by $\WD(\rho)=(r, N)$ the associated Weil--Deligne representation, where $r:W_K \to \GSp_4(\overline{\Q}_p)$, $N \in \mathfrak{gsp}_4(\overline{\Q}_p)$.

\subsection{Arthur's Classification for automorphic representations of $\GSp_4$} \label{s: Arthur's Classification}
We recall Arthur's classification \cite{Art04}, which in particular describes the $A$-parameters $\psi$ of each of the six general families of automorphic representations for $\GSp_4$ over a number field $F$ that occur in the discrete spectrum. Let $\nu(n)$ denote the irreducible representation of $\mathrm{SL}_2(\mathbb{C})$ of dimension $n$. An automorphic $\GL_n(\mathbb{A}_F)$-representation $\mu$ is said to be $\chi$-self dual for an id\`ele class character $\chi$ of $F$ if $\mu\simeq\mu^\vee\otimes\chi$. As much progress has been made on making Arthur's classification unconditional \cite{Ato+25}, dependent upon on cases of the twisted weighted fundamental lemma, we assume this.
\begin{enumerate}
    \item[(a)] Stable, semisimple (general type)
    \begin{equation*}
        \psi = \psi_1 = \mu \boxtimes 1,
    \end{equation*}
    where $\mu$ is a $\chi$-self dual, unitary cuspidal automorphic representation of $\GL_4(\mathbb{A}_F)$ that is not of orthogonal type.
    \item[(b)] Unstable, semisimple (Yoshida type)
    \begin{equation*}
        \psi = \psi_1 \boxplus \psi_2 = (\mu_1 \boxtimes 1) \boxplus (\mu_2 \boxtimes 1),
    \end{equation*}
    where $\mu_i$ are distinct unitary, cuspidal automorphic representations of $\GL_2(\mathbb{A}_F)$ whose central characters satisfy $\chi_{\mu_1} = \chi_{\mu_2} = \chi$.
    \item[(c)] Stable, mixed (Soudry type)
    \begin{equation*}
        \psi = \psi_1 = \mu  \boxtimes \nu(2),
    \end{equation*}
    where $\mu = \mu(\theta)$ is a unitary cuspidal automorphic representation of $\GL_2(\mathbb{A}_F)$ of orthogonal type with $\chi_\mu^2 = \chi$.
    \item[(d)] Unstable, mixed (Saito--Kurokawa type)
    \begin{equation*}
        \psi = \psi_1 \boxplus \psi_2 = (\lambda \boxtimes \nu(2)) \boxplus (\mu \boxtimes 1),
    \end{equation*}
    where $\lambda$ is in id\`ele class character of $F$ and $\mu$ is a unitary, cuspidal automorphic representation of $\GL_2(\mathbb{A}_F)$ with $\lambda^2 = \chi_\mu = \chi$.
    \item[(e)] Unstable, almost unipotent (Howe, Piatetski-Shapiro type)
    \begin{equation*}
        \psi = \psi_1 \boxplus \psi_2 = (\lambda_1 \boxtimes \nu(2)) \boxplus (\lambda_2 \boxtimes \nu(2)), 
    \end{equation*}
    where $\lambda_i$ are distinct id\`ele class characters of $F$ with $\lambda_1^2=\lambda_2^2=\chi$.
    \item[(f)] Stable, almost unipotent (one dimensional type)
    \begin{equation*}
        \psi = \psi_1 = \lambda \boxtimes \nu(4),
    \end{equation*}
    where $\lambda$ is an id\`ele class character of $F$ with $\lambda^4 = \chi$.
\end{enumerate}
\subsection{Representations of $\GSp_4$} \label{s:Mok}
Let $F$ be a totally real number field, $\A_F$ its ring of ad\`eles, and $G_F$ the absolute Galois group of $F$ with respect to a fixed algebraic closure $\overline{F}/F$.  For each place $v$ of $F$, let $L_{F_v}$ denote the $\mathrm{SL}_2$ form of the local Langlands group
\[
L_{F_v}:=\begin{cases}
  W_{F_v} & v \text{ archimedean},\\[4pt]
  W_{F_v}\times\mathrm{SL}_2(\C) & v \text{ non-archimedean}.
\end{cases}
\]
In particular, $L_\C=W_\C=\C^\times$ and $L_\R=W_\R=\C^\times\cup\C^\times j$, where $j^2=-1$ and
$j z j^{-1}=\overline{z}$ for $z\in\C^\times$. Fix a rational prime $p$. We define a class of representations of $\GSp_4(\A_F)$ that are expected to admit $p$-adic Galois representations, following \cite{BG14}.
\begin{defn} \label{def: aut rep of Galois type}
Let $\Pi$ be an automorphic representation of $\GSp_4(\A_F)$ with central character $\phi$ and let $S$ be the set of the places at which $\Pi$ is ramified together with those dividing $p$ and $\infty$. Then $\Pi$ is of \emph{Galois type} if $\Pi$ is discrete, the Hecke eigenvalues of $\Pi$ generate a number field\footnote{Although in general the Hecke eigenvalues of an automorphic representation may not generate a number field, those which are conjectured to have associated Galois representations are also conjectured to have Hecke eigenvalues generating a number field (cf. \cite{BG14}).} and there is a continuous semisimple Galois representation  $\rho_\Pi: G_F \to  \GSp_4(\overline{\Q}_p)$ that satisfies

\begin{itemize}
    \item $\nu \circ \rho_\Pi = \phi_p \epsilon^{-3}$, where $\phi_p$ is the $p$-adic Galois character associated to $\phi$, and there exists an integer $w$ such that for all archimedean places $v$ of $F$, the component of $\phi$ at $v$ is given by $\phi_v:a \mapsto a^{-w}$. 
    \item If $v \not \in S$, then $\rho_\Pi|_{G_{F_v}}$ is unramified, and for all $v$,  
    \begin{equation*}
         \iota_p \WD(\rho_\Pi|_{G_{F_v}})^{ss} \simeq \mathrm{rec}_{\mathrm{GT}}(\Pi_v \otimes |\nu|_v^{-3/2})^{ss}.
    \end{equation*}
    \item Write $\rho_\Pi$ for the composition of $\rho_\Pi$ with the inclusion $\mathrm{std} : \GSp_4(\overline{\Q}_p) \hookrightarrow \GL_4(\overline{\Q}_p)$. 
    If $v \vert p$, then $\rho_\Pi|_{G_{F_v}}$ is de Rham and, identifying embeddings $F \xhookrightarrow{}\overline{\mathbb{Q}}_p$ with the archimedean places of $F$ via $\iota_p:\overline{\mathbb{Q}}_p \cong \mathbb{C}$, the Hodge–Tate cocharacter associated to $\rho_\Pi|_{G_{F_v}}$ corresponds to the character $-\lambda_{\Pi, v_{\infty}} + (0,0;3)$ via the isomorphism in \S \ref{s: lie theory}, where $\lambda_{\Pi, v_{\infty}}$ is the infinitesimal character of $\Pi_{v_{\infty}}$. In particular, if $\lambda_{\Pi, v_{\infty}} = (a_v, b_v; c_v)$ with $a_v \geq b_v$ is $B$-dominant, then
    the Hodge--Tate weights at the embedding corresponding to $v_\infty \vert \infty$ are given by
    \begin{equation*}
        \{\delta_v, \delta_v + b_v, \delta_v + a_v, \delta_v + a_v + b_v\}, \quad \delta_v: = 1/2(-a_v-b_v + 3-c_v).
    \end{equation*}
    \item If $v \vert p$ and $\Pi_v$ is unramified, then $\rho_\Pi|_{G_{F_v}}$ is crystalline and $P^{\mathrm{cris}}_{\Pi,v}(X)=Q_{\Pi,v}(X)$, where $P^{\mathrm{cris}}_{\Pi,v}$ is the inverse characteristic polynomial of crystalline Frobenius on $\Dcris(\rho_\Pi|_{G_{F_v}})$ and $Q_{\Pi,v}$ is the inverse characteristic polynomial of geometric Frobenius $\Frob_v$ on $\mathrm{std} \circ \mathrm{rec}_{\mathrm{GT},p}(\Pi_v\otimes|\nu|_v^{-3/2})$.
\end{itemize}
\end{defn}

\begin{lemma}
\label{lem:gal-reps}
If $\Pi$ is cohomological and cuspidal, then $\Pi$ is of Galois type (\cite[Thm.~3.5]{Mok14}). The same is true if $\Pi$ is Saito--Kurokawa (Lemma~\ref{lem:SK-cohomological-structure}).
\end{lemma}

Let $\mu'$ be a cuspidal automorphic representation of $\PGL_2(\A_F)$ with global sign $\epsilon(1/2,\mu')$ such that $\mu_v'$ is holomorphic discrete series at all $v \vert \infty$, and let $S_\mathrm{Sch}$ be a nonempty finite set of places, disjoint from those dividing $\{p\}$, such that $\mu'_v$ is discrete series for each $v\in S_\mathrm{Sch}$ and $(-1)^{\# S_\mathrm{Sch}} = \epsilon(1/2,\mu')$. By Schmidt's construction \cite[Thm.~3.1]{Sch05}, there exists a $\PGSp_4(\A_F)$ representation defined by 
\begin{equation*}
    \Pi(\mu' \otimes \pi_{S_\mathrm{Sch}})=\bigotimes_v\Pi((\mu')_v \otimes \pi_{S_\mathrm{Sch}, v}), \qquad \pi_{S_\mathrm{Sch}, v} = \begin{cases}
    1_v & v \not \in S_\mathrm{Sch},\\[4pt]
    \mathrm{St}_v & v \in S_\mathrm{Sch}.
  \end{cases}
\end{equation*} 
We denote by $\pi$ the $\GSp_4(\mathbb{A}_F)$ lift of $\Pi(\mu' \otimes \pi_{S_\mathrm{Sch}})$ and call such a representation \emph{Saito--Kurokawa}. Its global $A$-parameter is $\psi =(\mu  \boxtimes 1)\boxplus(1_F \boxtimes \nu(2))$ in the sense of Arthur's classification, where $1_F$ denotes the trivial id\`ele class character of $F$ and $\mu$ is the cuspidal automorphic representation of $\GL_2(\mathbb{A}_F)$ lifting $\mu'$. The central character of $\pi$ is $\psi_\pi=\psi_\mu=1$. Such a representation $\pi$ is noncuspidal precisely when $S_\mathrm{Sch} = \emptyset$ and $L(1/2, \mu) \neq 0$. The SK points for which we prove complete rigidity are those whose underlying representation $\pi$ is noncuspidal.

If $\omega$ is an algebraic Hecke character of $F$ and $\nu$ is the similitude character of $\GSp_4$, then we also call the twist $\pi_\omega := (\omega\circ\nu)\otimes\pi$ Saito--Kurokawa, its global $A$-parameter is $\psi_\omega = (\mu_\omega\boxtimes 1) \boxplus (\omega\boxtimes\nu(2))$, where $\mu_\omega := \mu\otimes\omega$, and its central character is $\omega^2$. By the properties of $\mathrm{rec}_{\mathrm{GT}}$ and the shape of its global $A$-parameter, a $p$-adic Galois representation can be attached to $\pi_\omega$.

\begin{lemma}
\label{lem:SK-cohomological-structure}
Let $\pi_\omega$ be a Saito--Kurokawa representation of $\GSp_4(\A_\Q)$ with trivial central character, and global $A$-parameter $\psi_\omega = (\mu_\omega\boxtimes 1) \boxplus (\omega\boxtimes\nu(2))$ with $\mu_\omega = \mu\otimes\omega$. Then:
\begin{enumerate}
\item[(1)] There exists a continuous, semi-simple Galois representation $\rho_{\pi, \omega} : G_\Q \to \GL_4(\overline{\mathbb{Q}}_p)$ defined by 
\begin{equation*}
    \rho_{\pi, \omega}=\omega_p(-1)\oplus\omega_p(-2)\oplus\rho_{\mu, \omega},\qquad
\rho_{\mu, \omega}:=\widetilde{\rho}_{\mu, \omega}(-1),  
\end{equation*} 
where for each finite prime $v$,\[  \widetilde{\rho}_{\mu, \omega}|_{G_{\Q_v}}  \cong \mathrm{rec}_{F_v}\!\left(\mu_{\omega,v}\otimes \lvert\cdot\rvert_v^{-1/2}\right).\]
\item[(2)] The infinitesimal character of $\pi_{\omega, \infty}$ is $(k-1, k-2;0)$ for some $k \in \Z_{\ge 2}$ and the Hodge--Tate weights of $\rho_{\pi, \omega}$ at $p$ are $\{3-k, 1, 2, k\}$. In particular, $\rho_{\mu, \omega}$ has Hodge--Tate weights $\{3-k, k\}$ and corresponds to a holomorphic newform of weight $2k-2$.
\end{enumerate}
\end{lemma}
\begin{proof}
This follows from \cite[\S \S 3.1--3.2]{Mok14} and \cite[Thm.~3.1, \S 4]{Sch05}. 
\end{proof}
\begin{rem} \label{rem: F=Q}
Although Galois representations exist for (Saito--Kurokawa) representations over general totally real fields and the Bloch--Kato conjecture predicts that $H_f^1(F,\rho_{\mu,F}(2))$ is determined by the order of vanishing of $L(s,\rho_{\mu,F}(2))$ at its central value, our rigidity theorem is specific to the case $F=\Q$; the vanishing of $H_f^1(\Q,\Q_p(1))$, intrinsic to our argument, fails for totally real $F \neq \Q$, for which $H_f^1(F,\Q_p(1))\cong \mathcal{O}_F^\times\otimes\Q_p$ has positive dimension. 
\end{rem}

\subsection{$p$-adic Hodge and Bloch--Kato theory} \label{s: p-adic Hodge}
For a finite set $S$ of rational primes containing $p$, let $G_{\Q,S}$ denote the Galois group of the maximal extension of $\Q$ inside $\overline{\Q}$ unramified outside $S$. Let $\epsilon : G_\Q \to \Z_p^\times$ denote the $p$-adic cyclotomic character, and also its restriction to the decomposition group $G_{\Q_p}:=\Gal(\overline{\Q}_p/\Q_p)$ determined by $\overline{\Q} \xhookrightarrow{} \overline{\Q}_p$. Write $\Q_p(1)$ for the one-dimensional $G_{\Q_p}$-representation defined by~$\epsilon$ and for $W$ a representation of $G_{\Q_p}$ and $m\in\Z$, let $W(m) := W\otimes_{\Q_p}\epsilon^m.$ We use the convention that $\epsilon$ has Hodge--Tate weight $-1$. For a $p$-adic Galois representation $V$ unramified outside $S$ and $\star \in \{f, g\}$, we define
    \begin{align*}
        H_\star^1(\mathbb{Q}, V):= \ker \left( H^1(G_{\mathbb{Q},S}, V) \longrightarrow \bigoplus_{\ell \in S} \frac{H^1(\mathbb{Q}_\ell, V)}{H_\star^1(\mathbb{Q}_\ell, V)}\right),
    \end{align*}
    where 
    \begin{align*}
     H_f^{1}(\mathbb{Q}_\ell , V) :=& \left\{ \begin{array}{ll}
        H_\mathrm{ur}^1(\mathbb{Q}_\ell, V):=  \ker\left(H^1(\mathbb{Q}_\ell, V) \to H^1(I_\ell, V)\right)=H^1(G_{\mathbb{Q}_\ell}/I_\ell, V^{I_\ell}) & \ell \neq p\\
             \ker\left(H^1(\mathbb{Q}_p, V) \to H^1(\mathbb{Q}_p, V \otimes_{\mathbb{Q}_p} B_{\mathrm{cris}})\right) & \ell=p
        \end{array} \right. \\
        H_g^{1}(\mathbb{Q}_\ell , V): =& \left\{ \begin{array}{ll}
        H^1(\mathbb{Q}_\ell, V) & \ell \neq p \\
              \ker\left(H^1(\mathbb{Q}_p, V) \to H^1(\mathbb{Q}_p, V \otimes_{\mathbb{Q}_p} B_{\mathrm{dR}})\right) & \ell=p 
        \end{array} \right.
    \end{align*}
where $I_\ell$ is the inertia subgroup of $G_{\mathbb{Q}_\ell}$ and $B_\mathrm{cris}$, $B_\mathrm{dR}$ are the period rings defined in \cite{Fon94}. In particular, the group $H_f^1(\mathbb{Q}, V)$ is independent of the choice of $S \setminus \{p \}$.

\section{Local representation theory}
\label{s: Local representation theory}
Let $F$ be a non-archimedean local field and $\Omega$ an algebraically closed field. Let $G$ be a connected reductive algebraic group over $F$ and $P=MU \subset G$ a parabolic subgroup with Levi $M$ and unipotent radical $U$. If $\pi$ is an admissible $\Omega[P(F)]$-module, then denote by $\mathrm{Ind}_{P(F)}^{G(F)} \pi$ the smooth induction. For $\Omega= \mathbb{C}$, denote by $i_P^G \pi: = \mathrm{Ind}_{P(F)}^{G(F)} \pi\otimes \delta_P^{1/2}$ the normalised induction, where $\delta_P: P(F) \to \mathbb{R}_{>0}$, \ $\delta_P(x) = \lvert \det{(\mathrm{Ad}(x)) \vert_{\mathrm{Lie}(U)}}\rvert_F$ is the modulus character. Let $(\pi, V )$ be a representation of $G$. Let $(\rho, W)$ be a representation of $M$, which we extend to a representation of $P$ by letting $U$ act trivially. Let $\Alg(G)$ denote the category of smooth $G$-representations.

\begin{lemma} \label{lem: geometric lemma}
    Let $P, B \subset G= \GSp_4$ be the Siegel parabolic and upper triangular Borel. Then the semisimplification of the normalised Jacquet module $J_B(\pi)$ of a  parabolic induction $\pi= i_P^G \psi$ for a character $\psi \in \mathrm{Alg}(L(P))$ is 
    \begin{equation*}
        J_B(\pi)^{ss} \simeq \bigoplus_{w \in W(P)} (\psi' \cdot w), \qquad \psi'=\psi \delta_P^{-1/2}\vert_T, \qquad (\psi'\cdot w)(t) = \psi'(w t w^{-1}),
    \end{equation*}
where $W(P)= \{\mathrm{id}, s_2, s_2 s_1, s_2 s_1 s_2\}$ and $\delta_P = \delta_{B_2} \times \delta_{B_1}$, for $B_n$ the upper triangular Borel of $\GL_n$.
\end{lemma}
\begin{proof}
    This follows from the geometric lemma \cite[Geometrical Lemma]{BZ77} with $M=L(P) \simeq \GL_2 \times \GL_1$, $N=L(B)=T$ and $W(P) =  W^{L(P), T} = \{\mathrm{id}, s_2, s_2 s_1, s_2 s_1 s_2\}$.
\end{proof}

\subsection{Refinements of the Saito--Kurokawa representation}
Let $F$ be a non-archimedean local field of characteristic zero, and $P, B \subset \GSp_4$ be the Siegel parabolic and upper triangular Borel. For characters $\chi_1,\chi_2,\sigma$ of $F^\times$, let
\[
  \chi_1 \times \chi_2 \rtimes \sigma
  := i_B^{\GSp_4}(\chi_1\otimes\chi_2\otimes\sigma), \quad \chi_1 \otimes \chi_2 \otimes \sigma : \begin{pmatrix}
        t_1 & * & * & * \\
        & t_2 & * & * \\
        & & u t_2^{-1} & * \\
        & & & u t_1^{-1}
    \end{pmatrix}
\mapsto \chi_1(t_1) \chi_2(t_2) \sigma(u).
\]
If $\pi$ is a representation of $\GL_2(F)$ and $\sigma$ is a character of $F^\times$, let $\pi\rtimes\sigma$ be the representation of $\GSp_4(F)$ induced from the $P(F)$-representation
\begin{equation*}
    \begin{pmatrix}
        A & * \\ & uA'
    \end{pmatrix} \mapsto \sigma(u) \pi(A).
\end{equation*}
\begin{defn}
\label{def: refinement-local}
Let $F$ be a non-archimedean local field of characteristic zero, and let $\pi$ be an unramified representation of $\GSp_4(F)$.  An \emph{accessible refinement} of $\pi$ is a (necessarily smooth) character $\psi=\chi_1\otimes\chi_2\otimes\sigma$ appearing as a subquotient of $ J_B(\pi)^{T(\mathcal{O}_F)}$.
\end{defn}
\noindent
Our $p$-adic families interpolate systems of Hecke eigenvalues attached to automorphic $\GSp_4(\A_\Q)$-representations, unramified at a rational prime $p$, and of Galois type (Definition~\ref{def: aut rep of Galois type}). Such a system is determined by the global representation, together with an accessible refinement of $\Pi_p$.

\begin{prop}\label{prop: refinements}
Let $\omega : \Q^\times\backslash\A_\Q^\times \to \C^\times$ be an algebraic Hecke character with $p\nmid\mathrm{cond}(\omega)$ and set $\omega^p := \omega|_{\Q_p^\times}$.  Let $\pi_\omega = (\omega\circ\nu)\otimes\pi$ be a Saito--Kurokawa representation as in \S \ref{s:Mok}, unramified at $p$. Then \[
  \pi_{\omega,p}
  =
  L(\lvert\cdot\rvert_p^{1/2}\mu_p\rtimes\lvert\cdot\rvert_p^{-1/2}\omega^p) =\chi\,1_{\GL_2}\rtimes\chi^{-1}\omega^p, \quad \mu_p = i_{B_2}^{\GL_2}(\chi\otimes\chi^{-1}),
\]
where $\chi$, $\omega^p$ are unramified characters of $\Q_p^\times$ and $\chi \notin \{\lvert\cdot\rvert_p^{\pm 3/2},\ \lvert\cdot\rvert_p^{\pm 1}\xi\}$ for any quadratic character $\xi$ of $\Q_p^\times$, and $L(\cdot)$ is the Langlands quotient. The accessible refinements of $\pi_{\omega,p}$ are

\begin{multicols}{2}
\begin{enumerate}[leftmargin=1.7em]
\item $\psi_1= \chi \lvert\cdot\rvert_p^{-1/2} \otimes \chi \lvert\cdot\rvert_p^{1/2} \otimes \chi^{-1} \omega^p$
\item  $\psi_2= \chi \lvert\cdot\rvert_p^{-1/2} \otimes \chi^{-1} \lvert\cdot\rvert_p^{-1/2} \otimes \omega^p \lvert\cdot\rvert_p^{1/2}$
\item $\psi_3 =\chi^{-1} \lvert\cdot\rvert_p^{-1/2} \otimes \chi \lvert\cdot\rvert_p^{-1/2} \otimes \omega^p \lvert\cdot\rvert_p^{1/2}$
\item $\psi_4 = \chi^{-1} \lvert\cdot\rvert_p^{-1/2} \otimes \chi^{-1} \lvert\cdot\rvert_p^{1/2} \otimes \chi \omega^p$
\end{enumerate}
\end{multicols}
\noindent
They are pairwise distinct if and only if $\chi^2\neq1$ (if $\chi^2=1$, then $(\psi_1,\psi_2) = (\psi_4,\psi_3)$), in which case $J_B(\pi_{\omega,p})$ is semisimple. 
\end{prop}
\begin{proof}
The expression for $\pi_{\omega, p}$ follows by the proof of \cite[Lem.~2.2]{Sch05}, noting that, by \cite[Rem.~3.2(c)]{Sch05}, the twist by $\omega$ gives $ \pi_{\omega,p} = \chi\,1_{\GL_2}\rtimes\chi^{-1}\omega^p$, together with the irreducibility conditions $\chi \notin \{\lvert\cdot\rvert_p^{\pm 3/2},\lvert\cdot\rvert_p^{\pm 1}\xi\}$ for $\pi = \chi\,1_{\GL_2}\rtimes\sigma$ of \cite[Lem.~3.3, Lem.~3.7]{ST93}. Applying Lemma~\ref{lem: geometric lemma} to $\psi = \chi1_{\GL_2} \otimes \chi^{-1}\omega^p \in \Alg(L(P))$ gives precisely the four filtered pieces of $J_B(\pi_{\omega, p})$, noting that $s_1 \cdot [a, b; c] = [b, a; c]$ and $s_2 \cdot [a, b; c]=[a, cb^{-1};c]$ for $t=\mathrm{diag}(a,b,cb^{-1},ca^{-1})=[a,b;c]\in T$. Since the unramified characters $\chi$, $\omega^p$ and $\lvert\cdot\rvert_p$ are trivial on $T(\Z_p)$, these are precisely the accessible refinements of $\pi_{\omega, p}$. They are pairwise distinct if and only if $\chi^2\neq1$ (equivalent to the regularity of $\chi\lvert\cdot\rvert_p^{1/2}\otimes\chi\lvert\cdot\rvert_p^{-1/2}\otimes\chi^{-1}\omega^p$ by \cite[Lem.~3.3]{ST93}), in which case $J_B(\pi_{\omega,p})$ is multiplicity-free and hence semisimple. 
\end{proof}

\section{The Hecke algebra and Hecke operators}
\label{s: The Hecke algebra and Hecke operators}

\begin{defn}
\label{def: local Hecke algebra}
Let $A$ be a commutative ring, $v$ a finite place of a number field $F$, and $G/\mathcal{O}_{F_v}$ a
split reductive group. For each compact open subgroup $K_v \subset G(F_v)$, the \emph{local Hecke algebra}
$H(G(F_v),K_v,A)$ is the $A$-algebra of $K_v$ bi-invariant, locally constant, compactly supported functions $G(F_v)\to A$, with multiplication given by convolution.
\end{defn}

\subsection{Atkin--Lehner theory}
\label{subsec: AL}
We collect the results on Iwahori Hecke algebras that we require, following \cite[\S2.4]{Box+21}.  Let $v$ be a finite place of a number field $F$, with residue field $k(v)$ of cardinality $q_v$, and fix a uniformiser $\varpi_v\in\mathcal{O}_{F_v}$.  Let $G/\mathcal{O}_{F_v}$ be a split reductive group with Borel $B$ and maximal torus $T$. Let $E/\Q_p$ be a finite extension with ring of integers $\mathcal{O}_E$ and residue field $k$, and fix a square root $q_v^{1/2}\in E$. Let $\mathrm{Iw}(v)$ denote Iwahori of $G(F_v)$, the preimage of $B(k(v))$ in $G(\mathcal{O}_{F_v})$. Set $\mathcal{H}_v^{\mathrm{Iw}} := H(G(F_v),\mathrm{Iw}(v), \mathcal{O}_E)$ and $\mathcal{H}_v^{\mathrm{Iw}}[1/p] := 
\mathcal{H}_v^{\mathrm{Iw}} \otimes_{\mathcal{O}_E} \mathcal{O}_E[1/p]$. For $g\in G(F_v)$, write $[\mathrm{Iw}(v)g\mathrm{Iw}(v)]$ for the characteristic function of the double coset in $\mathcal{H}_v^{\mathrm{Iw}}$. Let $\Delta\subset X^\ast(T)$ be the set of simple roots of the Weyl group of $(G, T)$, and let
\[
T(F_v)^+
  = \{t\in T(F_v)\mid \alpha(t)\in\mathcal{O}_{F_v}\ \forall\,\alpha\in\Delta\}.
\]
\begin{defn}
\label{def:AL}
Let $U\subset T(F_v)$ be the subgroup of diagonal elements whose entries are integral powers of $\varpi_v$ and $U^- = U \cap T(F_v)^+$. The \emph{Atkin--Lehner monoid algebra} $\mathcal{A}_v^-\subset\mathcal{H}_v^{\mathrm{Iw}}$ is the $\mathcal{O}_E$-subalgebra generated by $[\mathrm{Iw}(v)u\mathrm{Iw}(v)]$ for $u\in U^-$. After inverting $p$, these elements become invertible in $\mathcal{H}_v^{\mathrm{Iw}}[1/p]$. The \emph{Atkin--Lehner algebra} $\mathcal{A}_v \subset \mathcal{H}_v^{\mathrm{Iw}}[1/p]$ is the $\mathcal{O}_E[1/p]$-subalgebra generated by $[\mathrm{Iw}(v)u\mathrm{Iw}(v)]^{\pm1}$ for $u\in U^-$. 
\end{defn}
\noindent 
The map $u\mapsto[\mathrm{Iw}(v)u\mathrm{Iw}(v)]$ extends uniquely to ring isomorphisms
\[\mathcal{O}_E[U^-]\xrightarrow{\sim}\mathcal{A}_v^-,
\qquad
\mathcal{O}_E[U]\xrightarrow{\sim}\mathcal{A}_v.
\]
Thus any $\mathcal{A}_v$-module may be viewed as a $U$-module, with $u\in U^-$ acting via $[\mathrm{Iw}(v)u\mathrm{Iw}(v)]$.

\begin{prop}[{\cite[Prop.~2.4.2]{Box+21}}]
\label{prop:H1-hom}
For $x,y\in T(F_v)^+$,
\[
  [\mathrm{Iw}(v)x\mathrm{Iw}(v)]\cdot[\mathrm{Iw}(v)y\mathrm{Iw}(v)]
  = [\mathrm{Iw}(v)xy\mathrm{Iw}(v)].
\]
\noindent
Hence there is a well-defined group homomorphism $\theta_v : T(F_v)\to \mathcal{H}_v^{\mathrm{Iw}}[1/p]^\times$ defined as follows: if $x=yz^{-1}$ with $y,z\in T(F_v)^+$, then
\[
  \theta_v(x)
  = \delta_B^{1/2}(y)[\mathrm{Iw}(v)y\mathrm{Iw}(v)]\cdot
    \delta_B^{-1/2}(z)[\mathrm{Iw}(v)z\mathrm{Iw}(v)]^{-1},
\]
where $\delta_B$ is the modulus character of $B$. The kernel of $\theta_v$ is $T(\mathcal{O}_{F_v})$. 
\end{prop}
\noindent
Note that the action of $T(F_v)$ restricts to give the action of $U$ defined above. If $\pi$ is a smooth admissible $G(F_v)$-representation, then the action of $\mathcal{H}_v^{\mathrm{Iw}}$ on $\pi^{\mathrm{Iw}(v)}$ together with $\theta_v$ equips $\pi^{\mathrm{Iw}(v)}$ with a natural $\overline{E}[T(F_v)]$-module structure: for $t\in T(F_v)$ and $\varphi\in\pi^{\mathrm{Iw}(v)}$, $t\cdot\varphi := \theta_v(t)\cdot\varphi$. Via the natural map $\pi\to J_B(\pi)$, the same construction endows $(J_B(\pi))^{T(\mathcal{O}_{F_v})}$ with a compatible $\overline{E}[T(F_v)]$-module structure. 

\begin{lemma}
\label{lem:Iw-Jac}
Let $\pi$ be a smooth admissible $\overline{E}$-representation of $G(F_v)$. Then there is a canonical isomorphism of $\overline{E}[T(F_v)]$-modules
\[
  \pi^{\mathrm{Iw}(v)} \;\simeq\; (J_B(\pi))^{T(\mathcal{O}_{F_v})}.
\]
\end{lemma}

\begin{proof}
This follows from \cite[Lem.~4.1.1, Prop.~4.1.4]{Cas74} and Proposition~\ref{prop:H1-hom}.
\end{proof}

\subsection{The Hecke operators for $\GSp_4$} \label{s: The Hecke operators}
We define the Hecke operators that generate the Hecke algebra for $\GSp_4$. Let $F$ be a number field and let $p$ be a rational prime. For each finite place $v$ of $F$, fix a uniformiser $\varpi_v$. Let $T$ and $B$ be the torus and Borel of \S\ref{s: lie theory}. For $t=[a,b;c]=\mathrm{diag}(a,b,cb^{-1},ca^{-1}) \in T(F_v)$, set
\[
m_{v,0}:=[\varpi_v,\varpi_v;\varpi_v^2],\qquad
m_{v,1}:=[\varpi_v,\varpi_v;\varpi_v],\qquad
m_{v,2}:=[\varpi_v^2,\varpi_v;\varpi_v^2].
\]

\subsubsection{The spherical Hecke operators}
Let $A$ be a commutative ring. For $v \nmid p$ the \emph{spherical Hecke operators} $T_{v,i}: = [\GSp_4(\mathcal{O}_{F_v})m_{v,i} \GSp_4(\mathcal{O}_{F_v})]$, the characteristic functions of the respective double cosets, generate the spherical Hecke algebra $\mathcal{H}_v:=H(G(F_v), \GSp_4(\mathcal{O}_{F_v}), A)$ by \cite[\S 3.1.5]{Pil20}. They are independent of the choice of uniformiser $\varpi_v$.
\subsubsection{Hecke parameters}
If $\pi$ is an unramified representation of $\GSp_4(F_v)$, then $\pi$ is a constituent of an unramified principal series and the characteristic polynomial of $\mathrm{rec}_{\mathrm{GT},p}(\pi \otimes |\nu|^{-3/2})(\mathrm{Frob}_v)$ is given by
\begin{equation*}
    Q_v(X) := X^4 - t_{v,1}X ^3 + (q_v t_{v,2} + (q^3_v+q_v)t_{v,0})X^2 -q^3_v t_{v,0} t_{v,1} X + q^6_v t^2_{v,0},
\end{equation*}
where $t_{v,i}$ is the $T_{v,i}$-eigenvalue on $\pi^{G(\mathcal{O}_{F_v})}$.

\begin{prop}[{\cite[Prop.~2.4.6]{Box+21}}] \label{prop: uv and Hecke}
    If $v$ is a finite place of a number field $F$ and $\pi$ is an irreducible constituent of an unramified principal series $\chi_1 \times \chi_2 \rtimes \sigma$ of $\GSp_4(F_v)$, then 
    \begin{equation*}
        \mathrm{rec}_{\mathrm{GT},p}(\pi)^{ss} = \sigma \circ \Art_{F_v}^{-1} \otimes \left((\chi_1\chi_2)\circ\Art_{F_v}^{-1} \;\oplus\; 
      \chi_1\circ\Art_{F_v}^{-1} \;\oplus\;
      \chi_2\circ\Art_{F_v}^{-1} \;\oplus\; 1\right),
    \end{equation*}
    where, for a finite extension $K/\mathbb{Q}_p$, $\Art_{K}$ denotes the Artin map $\Art_K: K^\times \xrightarrow{\sim} W^\mathrm{ab}_K$, normalised to send uniformisers to geometric Frobenii.
\end{prop}

\subsubsection{The Iwahori Hecke operators at $p$}
\label{s:Iwahori}
For $v\mid p$, let $K_v=\GSp_4(\mathcal{O}_{F_v})$ be a hyperspecial maximal compact subgroup containing the Iwahori subgroup $\Iw(v)$. We define the \emph{Iwahori Hecke operators}
\[
U_{v,i}=[\Iw(v)\,m_{v,i}\,\Iw(v)]\in\mathcal{H}_v^{\mathrm{Iw}},\qquad i=0,1,2.
\]
The $U_{v,i}$ generate $\mathcal{A}_v^-$, and, after inverting $p$, the $U_{v,i}$ and their inverses generate the Atkin--Lehner algebra $\mathcal{A}_v$. Let $\chi_1\otimes\chi_2\otimes\sigma$ be an accessible refinement of an unramified representation $\pi$ of $\GSp_4(F_v)$ (cf. Definition~\ref{def: refinement-local}). The operators $U_{v,i}$ act on $\pi^{\Iw(v)}$ via the isomorphism $\pi^{\Iw(v)} \simeq J_B(\pi)^{T(\mathcal{O}_{F_v})}$ of Lemma~\ref{lem:Iw-Jac}, and their eigenvalues $u_{v,i}$ are given by
\begin{equation*}
u_{v,i}=\delta_B(m_{v,i})^{-1/2}
(\chi_1\otimes\chi_2\otimes\sigma)(m_{v,i}).
\end{equation*}
By direct computation,
\begin{equation*}
u_{v,0}=(\chi_1\chi_2\sigma^2)(\varpi_v),\quad
u_{v,1}=q_v^{3/2}(\chi_1\chi_2\sigma)(\varpi_v),\quad
u_{v,2}=q_v^2(\chi_1^2\chi_2\sigma^2)(\varpi_v),
\end{equation*}
where $q_v$ is the cardinality of the residue field of $F_v$. In particular, since the roots of $Q_v(X)$ are $q_v^{3/2} \cdot \{(\chi_1\chi_2\sigma) (\varpi_v),(\chi_1\sigma)(\varpi_v),(\chi_2\sigma)(\varpi_v),\sigma(\varpi_v)\}$ by Proposition~\ref{prop: uv and Hecke}, an accessible refinement of $\pi$ determines an ordering of these roots
\begin{equation*}
    \bigl(u_{v,1},\ 
    u_{v,2} u_{v,1}^{-1} q_v,\ 
    u_{v,0} u_{v,2}^{-1}u_{v,1} q_v^2,\ 
    u_{v,0} u_{v,1}^{-1} q_v^3\bigr).
\end{equation*}

\begin{cor}
\label{cor:hecke-crys}
Let $\Pi$ be an automorphic representation of $\GSp_4(\A_\Q)$ of Galois type with central character $\phi$. If $\Pi_p$ is unramified, then an accessible refinement of $\Pi_p$ determines an ordering of the crystalline Frobenius eigenvalues of $\rho_\Pi|_{G_{\Q_p}}$, expressible in terms of the $u_{p, i}$-eigenvalues of $\Pi_p^{\Iw(p)}$.
\end{cor}
\begin{proof}
    This follows from the last point of Definition~\ref{def: aut rep of Galois type}.
\end{proof}

\section{\texorpdfstring{$p$}{p}-adic families} \label{s: family} 
\noindent
We adapt Bella\"iche's \cite{Bel10} definition of a $p$-adic family. 

\begin{defn}
\label{def: weight space}
Let $T$ and $Z_{\GSp_4}$ be the diagonal torus and centre of $\GSp_4$. The \emph{weight space} $\mathscr{W}$ is the rigid analytic space over $\Q_p$ such that $\mathscr{W}(\C_p)=\mathrm{Hom}_{\mathrm{cont}}\bigl(T(\Z_p),\C_p^\times\bigr)$. Let $\phi:\Q^\times\backslash\A_\Q^\times\to\C^\times$ be an algebraic Hecke character, and let $\tilde{\phi}:\Q^\times \backslash \A_\Q^\times\to\overline{\Q}_p^\times$ be the associated $p$-adic Hecke character\footnote{If $\phi$ is algebraic, then $\phi|_{\R_{>0}^\times}(z)=z^w$ for some $w\in\Z$, and $\tilde{\phi}(z)=\iota_p^{-1}(\phi(z)z_\infty^{-w})\,z_p^{\,w}$, where $z_\infty$ and $z_p$ denote the components of $z$ at $\infty$ and $p$ respectively. In particular, $\tilde{\phi}$ is continuous for the $p$-adic topology, since $\phi$ is continuous for the complex topology, so its local component $\tilde{\phi}_p:\Q_p^\times\to\overline{\Q}_p^\times$ is well-defined.}. 
We define the two-dimensional rigid subspace
\[
\mathscr{W}_\phi:
=
\{\chi\in\mathscr{W}:\chi|_{Z_{\GSp_4}(\Z_p)}=\tilde{\phi}_p[3]\}
\subset\mathscr{W}, \quad [3]: z \mapsto z^3
\]
\end{defn}

\begin{defn} \label{def: level and Hecke algebra}
Let $\A_f^p$ denote the finite ad\`eles away from $p$, and let $K^p=\prod_{\ell\neq p}K_\ell$ be a compact open subgroup of $\GSp_4(\A_f^p)$, the \emph{tame level}. Let $S$ be the finite set of primes $\ell$ such that either $\ell=p$ or $K_\ell$ is not maximal hyperspecial. Define the \emph{level} to be $K:=K^pK_p$, where $K_p=\GSp_4(\Z_p)$. The \emph{Hecke algebra} of level $K$ is $\mathcal{H}:=\left(\bigotimes_{\ell\notin S}\mathcal{H}_\ell\right)\otimes\mathcal{A}_p$, where $\mathcal{H}_\ell:=H(\GSp_4(\Q_\ell),K_\ell,\Z)$.
\end{defn}

\begin{defn} \label{def: refined rep}
A \emph{$p$-refined representation} $(\Pi,\psi_{\mathcal{R}})$ of level $K$ consists of a discrete automorphic representation $\Pi$ of $\GSp_4(\A_\Q)$ satisfying $\Pi^K\neq 0$, together with an accessible refinement $\psi_{\mathcal{R}}$ of $\Pi_p$. Such a pair determines a character $\psi_{\Pi,\mathcal{R}}= \iota_p^{-1}(\psi_\Pi)\otimes\psi_{\mathcal{R},p}:\mathcal{H}\to\overline{\Q}_p$\footnote{since $\Pi$ is of Galois type, $\psi_{\Pi}$ takes values in a number field so we may view $\psi_{\Pi}$ as a character $\mathcal{H}\to\overline{\Q}_p$ via $\iota_p:\overline
{\Q}_p \xrightarrow[]{\sim} \C$.}, where $\psi_\Pi:\bigotimes_{\ell\notin S}\mathcal{H}_\ell\to\C$ is the character giving the action of $\bigotimes_{\ell\notin S}\mathcal{H}_\ell$ on the one-dimensional space $\prod_{\ell\notin S}\Pi_\ell^{K_\ell}$ and $\psi_{\mathcal{R},p}:\mathcal{A}_p\to\overline
{\Q}_p$ is the character defined by $\psi_{\mathcal{R},p}\vert_U=\iota_p^{-1}(\psi_{\mathcal{R}}\cdot\delta_B^{-1/2}) \delta_{\underline{k}}$\footnote{The factor $\delta_B^{-1/2}$ arises from the definition of the map $T(\Q_p)\to\mathcal{H}_v^{\mathrm{Iw}}[1/p]^\times$ of Proposition~\ref{prop:H1-hom}.}, where, if $\lambda_{\Pi, \infty} = (a, b;c)$, then $\delta_{\underline{k}}= \lambda_{\Pi, \infty}- \delta_B^{1/2} = (a-2, b-1;c)$ (cf. Definition~\ref{def: aut rep of Galois type}).
\end{defn}

\begin{defn}
\label{def:p-adic-family-GSp4}
A \emph{$p$-adic family} $X(\phi)$ interpolating automorphic representations of $\GSp_4(\A_\Q)$ of Galois type and level $K$ with central character $\phi:\Q^\times\backslash\A_\Q^\times\to\C^\times$ consists of:
\begin{itemize}[itemsep=2pt,leftmargin=1.2em]
\item a reduced, separated, and irreducible rigid analytic space $X$ over $\Q_p$;
\item an algebraic Hecke character $\phi$, the \emph{central character} of $X(\phi)$, and \emph{similitude character} $\phi_p\epsilon^{-3}$ of $X(\phi)$, where $\phi_p$ is the $p$-adic Galois character associated to $\phi$; 
\item an analytic weight map $\kappa:X\longrightarrow\mathscr{W}_\phi$;
\item a ring homomorphism $\psi:\mathcal{H}\to\mathcal{O}(X)$;
\item a subset $Z\subset X(\overline{\Q}_p)$ that is Zariski dense and accumulates at all its points, such that:
\begin{itemize}[itemsep=1pt,leftmargin=0.6em]
\item[$\cdot$] for all $z\in Z$, the evaluation $\psi_z=\ev_z\circ\psi:\mathcal{H} \to \overline{\Q}_p$ is of the form $\psi_{\Pi,\mathcal{R}}$ for an automorphic representation $\Pi$ of $\GSp_4(\A_\Q)$ of Galois type, level $K$ and central character $\phi$, together with an accessible refinement $\psi_\mathcal{R}$ of $\Pi_p$; in this case we say that the family \emph{passes through} $(\Pi,\psi_{\mathcal{R}})$;
\item[$\cdot$] the map $Z\to \Hom(\mathcal{H},\overline{\Q}_p)$, $z\mapsto\psi_z$ is injective;
\item[$\cdot$] for all $z\in Z$, if $\psi_z=\psi_{\Pi,\mathcal{R}}$, then the associated Galois representation $\rho_z=\rho_{\Pi}$ is multiplicity-free with integer Hodge--Tate weights $\kappa_1(z) \leq \kappa_2(z)\leq\kappa_3(z)\leq\kappa_4(z)$ satisfying
\[
\kappa_1(z)+\kappa_4(z)=\kappa_2(z)+\kappa_3(z)=v_p(\phi(\eta_p))+3,
\]
and $\kappa(z)$, the \emph{algebraic weight} of $z$, coincides with the character $-\lambda_{\Pi_\infty} +(0, 0;3)$ associated to the Hodge--Tate cocharacter of $\Pi$ as in Definition~\ref{def: aut rep of Galois type};
\end{itemize}
\item a continuous four-dimensional pseudocharacter $T:G_{\Q,S}\to\mathcal{O}(X)$ such that for all $z\in Z$, $\rho_z$ is the semisimple representation of $G_{\Q,S}$ of trace $T_z = \ev_z \circ T$.
\end{itemize}
\end{defn}
 
\begin{rem} \label{rem: simplifying assumptions}
\begin{enumerate}[itemsep=1pt,leftmargin=1.3em]
    \item We assume that $X$ is reduced, so that all points are closed and local rings are reduced; replacing $X$ by its reduced subspace $X^{\mathrm{red}}$ does not change its dimension, so it suffices for rigidity to consider reduced $X$.
    \item Similarly we assume that $X$ is irreducible; if not, then our arguments can be applied to each irreducible component to deduce rigidity. 
    \item A subset $Z\subset X$ \emph{accumulates} at $x\in X$ if, for every open neighbourhood $U$ of $x$ there exists an affinoid neighbourhood $V\subset U$ of $x$ such that $Z\cap V$ is Zariski dense in $V$ (cf.\ \cite[\S3.3.1]{BC09}). In particular, if $Z$ accumulates at some $z_0\in Z$ and $X$ is irreducible, then $Z$ is Zariski dense in $X$\footnote{Indeed, if $Z$ accumulates at $z_0$, then $Z \cap U$ is Zariski dense in some open neighbourhood $U$ of $z_0$. Since $X$ is irreducible, the closure of $U$ in $X$ is all of $X$, so $X = \overline{U} \subseteq \overline{Z} \subseteq X$, and hence $Z$ is Zariski dense in $X$.}.
    \item The injectivity of $Z\to\Hom(\mathcal{H},\overline{\Q}_p)$ ensures that $X(\phi)$ parametrises distinct systems of Hecke eigenvalues; in particular, it excludes the case in which all points of $Z$ give the same eigensystem. Fixing the central character rules out $p$-adic families of twists. In this sense, $X(\phi)$ is nontrivial.
    \item  In practice Galois pseudocharacters exist on eigenvarieties (cf. \cite[Prop.~7.5.4]{BC09}, \cite[Prop.~7.1.1]{Che04}).
\end{enumerate}
\end{rem}

\begin{defn} \label{def: SK point}
A point $z=(\Pi,\psi_{\mathcal{R}})\in Z$ is a \emph{Saito--Kurokawa (SK) point} if $\Pi$ is a Saito–Kurokawa representation.
\end{defn}
\begin{lemma}
\label{lem: weight space for trivial CC fam}
The weight space of a $p$-adic family with trivial central character is $\mathscr{W}_1=\{\chi\in\mathscr{W}:\chi|_{Z_{\GSp_4}(\Z_p)}=(t\mapsto t^3)\}$ and if $z$ is an SK point with trivial central character, then $\kappa(z) = (1-k, 2-k; 3)$ for $k \in \mathbb{Z}_{\ge 2}$ as in Lemma~\ref{lem:SK-cohomological-structure}.
\end{lemma}
\subsection{The analytic functions}
The images of the Iwahori Hecke operators $U_{p,i}$ under $\psi:\mathcal{H}\to\mathcal{O}(X)$ define analytic functions on $X$ that interpolate rescaled crystalline Frobenius eigenvalues of $\rho_z$ for all $z\in Z$.
\begin{prop}
\label{prop:Fi}
Let $X(\phi)$ be a $p$-adic family as in Definition~\ref{def:p-adic-family-GSp4}. Then there exist invertible analytic functions $F_1,F_2,F_3,F_4\in\mathcal{O}(X)^\times$ such that for all $z\in Z$, the crystalline Frobenius eigenvalues of $\rho_z$ at $p$ are
\[
\bigl(F_1(z)p^{\kappa_1(z)},F_2(z)p^{\kappa_2(z)},
F_3(z)p^{\kappa_3(z)},F_4(z)p^{\kappa_4(z)}\bigr)
\]
in the order $\mathcal{R}$ determined by the accessible refinement $\psi_{\mathcal{R}}$ of $\Pi_p$ at $p$. 
\end{prop}
\begin{proof}
Since the $U_{p,i}$ are invertible in $\mathcal{A}_p$ and $\psi:\mathcal{H} \to \mathcal{O}(X)$ is a ring homomorphism, the functions
\[
F_1:=\psi(U_{p,1}),\quad
F_2:=\psi(U_{p,2})\psi(U_{p,1})^{-1},\quad
F_3:=\psi(U_{p,0})\,F_2^{-1},\quad
F_4:=\psi(U_{p,0})\,F_1^{-1}
\]
are analytic and invertible on $X$. The assertion on the crystalline Frobenius eigenvalues of $\rho_z$ follows from Definitions~\ref{def: refined rep} and \ref{def:p-adic-family-GSp4} and Corollary~\ref{cor:hecke-crys}.
\end{proof}

\subsection{Symmetry of $T$}
Let $A$ be a commutative ring. For a pseudocharacter $T : G_{\Q,S} \to A$, we define its dual by $T^\vee(g) := T(g^{-1})$. If $T=\mathrm{tr}(\rho)$ for a representation $\rho$, then $T^\vee=\mathrm{tr}(\rho^\vee)$ is the trace of the dual representation. For a character $\psi : G_{\Q,S} \to A$, define $T^\tau := T^\vee \otimes \psi$, where $(T\otimes\psi)(g)=T(g)\psi(g)$. This is an involution that preserves dimension on the set of $\psi$‑self‑dual pseudocharacters of $G_{\Q, S}$ over~$A$. If $\rho$ is $\psi$‑self‑dual and $T=\mathrm{tr}(\rho)$, then $T^\tau=T^\vee\otimes\psi=\mathrm{tr}(\rho^\vee\otimes\psi)=\mathrm{tr}(\rho)=T$. 

\begin{lemma} \label{lem: pseudo involution}
Let $T : G_{\Q,S} \to \mathcal{O}(X)$ be the pseudocharacter of a $p$-adic family $X(\phi)$. Then $T = T^\tau$, where $T^\tau=T^\vee\otimes\psi$ for $\psi=\phi_p(-3)$ the similitude character of $X(\phi)$.
\end{lemma}
\begin{proof}
For all $z=(\Pi,\psi_{\mathcal{R}})\in Z$, the Galois representation $\rho_z$ is $\phi_p(-3)$‑self‑dual. Hence $T_z = T_z^\vee \otimes \psi = T_z^\tau$ for all $z \in Z$. Since $T$ is determined by its specialisations at the Zariski dense subset $Z$ of $X$, the identity $T=T^\tau$ holds on all of $X$. 
\end{proof}

\begin{rem}\label{rem: tau-basechange}
The identity $T=T^\tau$ of Lemma~\ref{lem: pseudo involution} extends to $T\otimes_{\mathcal{O}(X)}A$ for any local $\mathcal{O}(X)$-algebra $A$. Twisting by any character $\chi:G_{\Q,S} \to A$ also preserves the $\tau$-structure on constituents. 
\end{rem}

\section{Rigidity for families of SK points}
\label{s: Rigidity for families of SK points}
If $X(\phi)$ is a $p$-adic family through an SK point $z_0=(\pi_\omega,\psi_{\mathcal R})\in Z$ in the sense of Definition~\ref{def: SK point}, then by twisting by $\omega^{-1}\circ\nu$ we may assume that $\omega=\phi=1$. Henceforth we write $v=v_p$ and fix such an SK point $z_0 = (\pi,\psi_i)$, so that $z_0$ is crystalline and has algebraic weight $(k(z_0)-1, k(z_0)-2;3)$, for $k(z_0) \in \Z_{\geq 2}$ and define $\alpha(z_0):= \,p^{k(z_0)-3/2} \chi(p)$, where $\pi_{p} = \chi 1_{\GL_2} \rtimes \chi^{-1}$. We analyse $p$-adic families in which a subset of SK points accumulates at $z_0$ and refer to such a $p$-adic family as an \emph{SK family}. 

If $z \in X(1)$ is an SK point with $\pi_{\omega,p} = \chi_z 1_{\GL_2}\rtimes \chi^{-1}_z \omega^p$ and $\chi_z(p)=\alpha(z)\,p^{3/2-k(z)}$, by the symmetry $\chi_z \leftrightarrow \chi_z^{-1}$ we may assume that $v(\alpha_z) \leq k(z)-3/2$ (cf. Remark~\ref{rem: eval conj}): this fixes a choice of $\chi_z$ over $\chi_z^{-1}$, and hence generically determines the refinement labellings $(\psi_1,\psi_4)$ and $(\psi_2,\psi_3)$ in Table~\ref{tab: 1}\footnote{for the case $v(\alpha_z) = k(z)-3/2$, the labelling is immaterial since the valuations $v(F_i(z))$ for $\psi_1$ and $\psi_4$ (resp.~$\psi_2$ and $\psi_3$) coincide and the condition on $C(z_0)$ if $i=3$ is unsatisfiable.}. The refinement $\psi_i$ plays a decisive role. The next proposition summarises the corresponding rigidity results.

\begin{prop}
\label{prop: SK summary}
Suppose that $X(1)$ is a $p$-adic family through a $\psi_i$-refined SK point $z_0$ of algebraic weight $(k(z_0)-1, k(z_0)-2;3)$ and $i\in\{1,2,3,4\}$. Let $Z_{\mathrm{SK}}\subset Z$ denote the set of all SK points.
\begin{enumerate}
\item[(i)] If $i\in\{2,3\}$ set $Z_{\psi_i}= Z_{\mathrm{SK}}$, and assume $C(z_0) := v(\alpha(z_0))\neq k(z_0)-2$ when $i=3$. 

\item[(ii)] If $i\in\{1,4\}$, let

\[
C(z_0):=v(F_1(z_0))=
\begin{cases}
v(\alpha(z_0)) & i=1\\[4pt]
2k(z_0)-3-v(\alpha(z_0)) & i=4
\end{cases}
\]
\noindent
Let $Z_{\psi_i}\subset Z_{\mathrm{SK}}$ be the subset of SK points such that $v(F_1(z))\neq C(z_0)$ if $z$ is $\psi_1$- or $\psi_4$-refined.
\end{enumerate}
If $Z_{\psi_i}$ accumulates at $z_0$, then $X(1)$ is a point.
\end{prop}

\subsection{SK points}
Let $z=(\pi_\omega,\psi_i)\in Z$ be an SK point with trivial central character and algebraic weight $(k-1,k-2;3)$, where $k=k(z) \in \Z_{\ge 2}$. Then $\pi_{\omega,p}$ is of the form $\pi_{\omega,p} = \chi \rtimes \chi^{-1} \omega^p$, $\rho_z=\omega_p(-1) \oplus \omega_p (-2) \oplus \rho_{\mu, \omega}$, and the values of the $F_i$ may be computed from Propositions~\ref{cor:hecke-crys} and \ref{prop:Fi}. Since $\omega^2=1$, we have $\omega(\eta_p)^2=1$ and hence $v(\omega(\eta_p))=0$, so the $v(F_i(z))$ are independent of $\omega$. Writing $\chi(p)=\alpha\,p^{3/2-k}$, we record these values in the following table.
\begin{equation} \label{tab: 1}
    \begin{array}{r|c|c|c|c}
    & \psi_1 & \psi_2 & \psi_3 & \psi_4  \\
    \hline
    F_1(z) & \omega(\eta_p) \alpha & \omega(\eta_p) p^{k-1} & \omega(\eta_p) p^{k-1} & \omega(\eta_p) \alpha^{-1} p^{2k-3} \\
    F_2(z) & \omega(\eta_p)p & \omega(\eta_p) \alpha p^{2-k} & \omega(\eta_p)\alpha^{-1} p^{k-1} & \omega(\eta_p) p
    \end{array}
\end{equation}
\begin{rem} \label{rem: eval conj}
If $\pi_{\omega}$ is unramified at $p$, then the crystalline Frobenius eigenvalues of $\rho_{\pi, \omega}$ are $\{\alpha\omega(\eta_p)p^{3-k},\omega(\eta_p)p^2, \omega(\eta_p)p, \alpha^{-1}\omega(\eta_p)p^{k}\}$, and those of $\rho_{\mu, \omega}$ are $\alpha_1:=\alpha \omega(\eta_p) p^{3-k}$ and $\alpha_2 :=\alpha^{-1}\omega(\eta_p)p^k$. We label them so that $v(\alpha_1)\le v(\alpha_2)$; this fixes a choice of $\chi$ over $\chi^{-1}$, and hence determines the refinement pairings $(\psi_1,\psi_4)$ and $(\psi_2,\psi_3)$ in Table~\ref{tab: 1}.
In particular, crystalline Frobenius eigenvalues of $\rho_{\pi, \omega}$ are repeated precisely when $\rho_{\mu, \omega}$ has repeated crystalline Frobenius eigenvalues\footnote{for non‑CM classical modular forms of weight $m \in \Z_{\ge 2}$, the crystalline Frobenius eigenvalues are conjectured (and known for $m=2$) to be distinct (\cite{CE98}).}. The crystalline Frobenius eigenvalues of $\rho_{\mu, \omega}$ are always distinct from $\omega(\eta_p)p$ and $\omega(\eta_p)p^2$, which suffices for our purposes. In particular, the rigidity arguments of \S \ref{s: Rigidity for families of SK points} depend only on the valuations of $F_1$ and $F_2$ in Table~\eqref{tab: 1} and therefore do not require the $\psi_i$ to be pairwise distinct. The same is true for Theorem~\ref{thm: X is a point} under the additional hypotheses of \S\ref{s:tot-rig}.  
\end{rem}

\subsection{Results}
We now prove the rigidity statements for SK families. 
\begin{prop} \label{prop:ref 2/3}
Suppose that $X(1)$ is a $p$-adic family through a $\psi_i$-refined SK point $z_0$ for $i\in\{2,3\}$, and assume that $v(\alpha(z_0))\neq k(z_0)-2$ when $i=3$. Let $Z_{\mathrm{SK}}\subset Z$ denote the subset of SK points. If $Z_{\mathrm{SK}}$ accumulates at $z_0$, then $X(1)$ is a point.
\end{prop}
\begin{proof}
If $\dim X>0$, then by analyticity of the $F_i$, on a sufficiently small connected affinoid neighbourhood $U$ of $z_0$ the valuations $v(F_1)$ and $v(F_2)$ are constant. Since $Z_{\mathrm{SK}}$ accumulates at $z_0$ and $U$ is an affinoid open of the irreducible space $X$, $U\cap Z_{\mathrm{SK}}$ is infinite. Since there are only finitely many automorphic forms of a fixed weight and level, and the algebraic weight of an SK point $z$ with trivial central character is $(k(z)-1,k(z)-2;3)$, it follows that $U\cap Z_{\mathrm{SK}}$ contains a point $z\neq z_0$ with $k(z)\neq k(z_0)$. If $z$ is $\psi_2$- or $\psi_3$-refined, then by Table~\eqref{tab: 1},
\[
k(z)-1=v(F_1(z))=v(F_1(z_0))=k(z_0)-1,
\]
contradicting $k(z)\neq k(z_0)$. Thus $z$ must be $\psi_1$- or $\psi_4$-refined. In this case,
\begin{equation} \tag{$\star$}
        1= v(F_2(z))=v(F_2(z_0))= \begin{cases}
            v(\alpha(z_0))+2-k(z_0) & i=2 \\
            -v(\alpha(z_0))+k(z_0)-1 & i=3
        \end{cases}
\end{equation}
Since $v(\alpha(z_0))\le k(z_0)-3/2$, if $i=2$, then $(\star)$ gives $1=v(F_2(z_0))=v(\alpha)+2-k(z_0) \leq 1/2$, a contradiction. If $i=3$, then $(\star)$ gives $v(\alpha(z_0))=k(z_0)-2$, contradicting the hypothesis on $z_0$. Hence $\dim X=0$, and since $X$ is irreducible, $X(1)$ is a point.
\end{proof}

\begin{prop} \label{prop: ref 1/4}
Suppose that $X(1)$ is a $p$-adic family through a $\psi_i$-refined SK point $z_0$ for $i \in \{1, 4\}$ and set $C=C(z_0):= v(F_1(z_0))$. Let $Z_{{\psi_i}} \subset Z$ denote the subset of SK points such that for all $z \in Z_{\psi_i} \setminus \{z_0\}$:
    \begin{enumerate}
        \item If $z$ is $\psi_1$-refined, then $v(\alpha(z)) \neq C$.
        \item If $z$ is $\psi_4$-refined, then $2k(z)-3 -v(\alpha(z))\neq C$. 
    \end{enumerate}    
If $Z_{\psi_i}$ accumulates at $z_0$, then $X(1)$ is a point. Moreover, the same conclusion follows if $Z_{\psi_i}$ is replaced by the subset $Z_{\psi_i}'=\{\,z\in Z_{\psi_i} : \text{neither slope of } f_z \text{ equals }C\,\}.$
\end{prop}
\begin{proof}
This follows by a similar argument to Proposition~\ref{prop:ref 2/3}, by choosing a point $z\in U\cap Z_{\psi_4}\setminus\{z_0\}$ with $k(z)\neq C+1$ and considering $v(F_1)$. The final assertion follows as the slopes of $f_z$, the $p$-adic valuations of the crystalline Frobenius eigenvalues of $\rho_{f_z} = \rho_{\mu, \omega}(3-k)$, are precisely $v(\alpha(z))$ and $-v(\alpha(z))+2k(z)-3$.
\end{proof}

For appropriate constants $M$, no nontrivial fixed-slope SK families can pass through a $\psi_1$- or $\psi_4$-refined SK point, where the \emph{slope} $v(\alpha(z))$ of an SK point $z$ is by definition the slope $\mathrm{min}(v(\alpha(z)), -v(\alpha(z))+2k(z)-3)$ of $f_z$.

\begin{prop}\label{prop: fixed slope}
Suppose that $X(1)$ is a $p$-adic family through a $\psi_i$-refined SK point for $i \in \{1, 4\}$ $z_0$ and set $C=C(z_0):=v(F_1(z_0))$. Let $M \neq C$ be a non-negative constant and let $Z_{\psi_i}''\subset Z_{\mathrm{SK}}$ denote the set of SK points with slope $v(\alpha(z))=M$. If $Z_{\psi_i}''$ accumulates at $z_0$, then $X(1)$ is a point.  
\end{prop}
\begin{proof}
This follows by the same argument as Proposition~\ref{prop: ref 1/4} with the extra assumption that $k(z) \neq (C+M+3)/2$ to derive a contradiction if $z$ is $\psi_4$-refined, noting that condition~(1) in loc. cit. is automatically satisfied.
\end{proof}

\begin{rem}\label{rem: semiordinary family}
We expect that in general $\psi_1$- (resp. $\psi_4$-) refined points deform in one-dimensional 
SK families, containing points excluded from $Z_{\psi_i}$ in assumption~(1) (resp.(2)) of Proposition~\ref{prop: ref 1/4}, for which $F_1$ interpolates the $U_p$-eigenvalue of the $(\omega(\eta_p)\alpha)(z)$ (resp. $(\omega(\eta_p) \alpha^{-1})(z) p^{2k(z)-3}$) refinement of $f_{z}$ for $\psi_1$- (resp. ($\psi_4$-) refined SK points. Indeed, when $v(\alpha(z_0))=0$, the modular form $f_{z_0}$ is $p$-ordinary, and the associated $\psi_1$-refined SK point can be interpolated in a one-dimensional ``semi-ordinary" family \cite[Prop.~4.2.5]{SU06}--assumption~(1) in Proposition~\ref{prop: ref 1/4} excludes the possibility that $X(1)$ is such a family. 
\end{rem}

\section{Pseudocharacters} \label{s:Pseudocharacters and the form of $T$}
We recall the general theory of pseudocharacters, following Taylor \cite{Tay91} and Bella\"iche--Chenevier \cite[\S1]{BC09}. A pseudocharacter is a central function $T : R \to A$ satisfying the $n$-dimensional pseudocharacter identity $S_{n+1}(T)=0$ for some integer $n$, where $R$ is a finitely generated $A$-algebra in which $n!$ is invertible and $A$ is a commutative ring (Definition~\ref{def: pseudochar}). The minimal such $n$ is the \emph{dimension} of $T$. 
\subsection{A key lemma}
Following Chenevier~\cite{Che11}, we recall the geometric structure of the rigid analytic spaces $X_n(G)$ that parametrises $n$-dimensional continuous semisimple representations of a profinite group $G$. This structure describes the reducible loci in $X_n$ and, via the universal property of $X_n$, allows us to control reducibility in the $p$-adic families that we study.

Let $G$ be a profinite group and $n \in \Z_{\ge 1}$. Assume that for every open subgroup $H \subset G$, the set of continuous homomorphisms $H \to \F_p$ is finite. Under this hypothesis, $X_n = X_n(G)$ exists as a finite-dimensional rigid analytic space over $\Q_p$ in the sense of Tate and represents the functor sending an affinoid $\Q_p$-algebra $S$ to the set of continuous $n$-dimensional pseudocharacters $G \to S$ in the sense of Taylor--Rouquier (cf.~\cite{Che11}). Concretely, the $S$-valued points of $X_n$ are precisely the continuous pseudocharacters $G \to S$ of dimension $n$.

The $\overline{\Q}_p$-points of $X_n$ correspond to conjugacy classes of continuous semisimple representations $G \to \GL_n(\overline{\Q}_p)$. For $x \in X_n$, let $k(x)$ denote its residue field and $\rho_x : G \to \GL_n(\overline{k(x)})$ the associated semisimple representation. If $\mathrm{Tr} : G \to \mathcal{O}(X_n)$ denotes the universal continuous pseudocharacter of dimension $n$, then $\mathrm{Tr}_x = \mathrm{tr}(\rho_x)$.

Let $r$, $a_1,\ldots,a_r \in \Z_{\ge 1}$ with $\sum a_i = n$. If $S$ is an affinoid $\Q_p$-algebra and $T_i : G \to S$ are continuous pseudocharacters of dimensions $a_i$, then $T_1 + \cdots + T_r$ is a continuous pseudocharacter of dimension $n$. Since this construction is functorial in $S$, it defines a morphism of rigid spaces 
\[
  \iota_{\underline{a}} :
  X_{a_1} \times \cdots \times X_{a_r} \longrightarrow X_n,
  \qquad \underline{a} = (a_1,\ldots,a_r).
\]
Set-theoretically,
\[
  X_n \setminus X_n^{\mathrm{irr}} = \bigcup_{\underline{a} \in A}
  \iota_{\underline{a}}(X_{a_1} \times \cdots \times X_{a_r}),
\]
where $A=\{\underline{a}=(a_1,\ldots,a_r)\mid r \in \Z_{> 1},a_i\in\Z_{\ge1},\ \sum a_i=n\}$ and $X_n^{\mathrm{irr}}$ is the set of irreducible points in the sense of Definition~\ref{def:a-red}. To describe the local structure of $X_n$ near reducible points and to control the corresponding reducible loci, we first generalise a technical lemma of Chenevier to $r \in \Z_{\geq 1}$.

\begin{defn} \label{def: pseudochar}
Let $A$ be a commutative ring and $R$ an $A$-algebra. Let $T:R \to A$ be an $A$-linear map which is central such that $T(xy)=T(yx)$ for all $x, y \in R$. For $n \in \Z_{\geq0}$, define a map $S_n(T):R^n \to A$ by $S_0(T):=1$ and
\begin{equation*}
    S_n(T)(x) := \sum_{\sigma \in S_n} \varepsilon(\sigma) T^\sigma(x), \qquad n \ge 1,
\end{equation*}
where $T^\sigma:R^n \to A$ is defined as follows. For $x = (x_1, \ldots, x_n) \in R^n$ and $\sigma=(j_1,\ldots ,j_m)$ a cycle, set $T^\sigma(x)=T(x_{j_1}, \ldots, x_{j_m})$\footnote{this is well-defined by \cite[\S 1.2.1]{BC09}.}. For $\sigma \in S_n$ with cycle decomposition $\sigma = \prod_{i=1}^r \sigma_i$ (including $1$-cycles), set $T^\sigma(x) = \prod_{i=1}^r T^{\sigma_i}(x)$. \par
$T$ is a \emph{pseudocharacter} on $R$ if there exists an integer $n$ such that $S_{n+1}(T)=0$ and $n!$ is invertible in $A$. The minimum such $n$ is the \emph{dimension} of $T$, and $T$ satisfies $T(1)=n$.
\end{defn}

\begin{defn} \label{def:a-red}
If $k$ is a field and $T$ a $k$-valued pseudocharacter of dimension $n$, then $T$ is \emph{$\underline{a} = (a_1, \ldots, a_r)$-reducible} if $T \otimes_k \overline{k}$ is the trace of a $\overline{k}$-representation $\rho = \rho_1 \oplus \cdots \oplus \rho_r$, where $\dim \rho_i = a_i$. If moreover the $\rho_i$ are irreducible (resp. pairwise non-isomorphic), then $T$ is \emph{precisely $\underline{a}$-reducible} (resp. \emph{multiplicity-free}). A point $\underline{u} \in X_n$ is (precisely) $\underline{a}$-reducible (resp. multiplicity-free) if the same is true for its conjugacy class of representations.
\end{defn}

\begin{lemma}[{\cite[Lem.~1.1]{Che11}}] \label{lem: 1.1}
Let $r, a_1,\ldots,a_r \in \Z_{\ge 1}$ with $\sum a_i = n$, and let $\underline{u} = (u_1,\ldots,u_r) \in X_{a_1} \times \cdots \times X_{a_r}$ correspond to a precisely $\underline{a}$-reducible and multiplicity-free representation $\rho_{\underline{u}} = \bigoplus_{i=1}^r \rho_{u_i}$. Let $L/k(x)$ be a finite extension over which $\rho_{\underline{u}}$ is defined, where $x = \iota_{\underline{a}}(\underline{u})$ is the image of $\underline{u}$ in $X_n$, and let $x' \in X_{n,L}(L)$ denote the corresponding point lying over $x$, where $X_{n, L}= X_n \times_{\Q_p} L$.  Then there exists an affinoid open subset $U\subset X_{n,L}$ containing $x'$ and a closed analytic subset $U_{\underline{a}}$ of $U$ whose points are exactly the $\underline{a}$-reducible points in $U$. Moreover all these points are precisely $\underline{a}$-reducible and $\mathrm{Tr}|_{U_{\underline{a}}}$ is a sum of $r$ $\mathcal{O}(U_{\underline{a}})$-valued pseudocharacters of dimensions $a_1, \ldots, a_r$.
\end{lemma}

\begin{proof}
Let $\underline{a}=(a_1,\ldots,a_r)$ and $x=\iota_{\underline{a}}(\underline{u})$. Choose a finite extension $L/k(x)$ over which each $\rho_{u_i}$ is defined, which determines a point $x' \in X_{{n,L}}(L)$ lying over $x$ (since this corresponds to choosing an embedding $\tau:k(x) \xhookrightarrow{} L$), and let $A= \mathcal{O}_{X_{n,L},x'}$ be the local ring of $X_{n, L}$ at $x'$, with residue field $L$. Let $S=\mathcal{O}(X_n)[G]/\langle \mathrm{CH}_n(T)\rangle$ be the quotient of $\mathcal{O}(X_n)[G]$ by the two‑sided ideal generated by the degree‑$n$ Cayley–Hamilton relation for $\mathrm{Tr}$ (cf.~\cite[Ex.~1.2.4(i)]{BC09}). Then $\mathrm{Tr}$ factors through a pseudocharacter $S\to\mathcal{O}(X_n)$ of dimension $n$, and $S$ satisfies the Cayley--Hamilton identity of degree $n$. By \cite[Thm.~1.4.4]{BC09}, the algebra $S_A := S\otimes_{\mathcal{O}(X_n)} A$ is a GMA of type $(a_1,\ldots,a_r)$, since $\rho_{\underline{u}}$ is multiplicity‑free, $L$-valued, and precisely $\underline{a}$‑reducible. Thus there exists an $A$‑algebra morphism 
\[
  \psi_A : M_{a_1}(A)\times\cdots\times M_{a_r}(A)\longrightarrow S_A
\]
such that $\mathrm{Tr}\circ\psi_A$ is the usual matrix trace. Since $A$ is the inductive limit over affinoid neighbourhoods of $x'$ in $X_{n,L}$, after shrinking to such a neighbourhood $U$ the morphism $\psi_A$ descends to a morphism of $\mathcal{O}(U)$‑algebras
\[
  \psi_U : M_{a_1}(\mathcal{O}(U))\times\cdots\times
  M_{a_r}(\mathcal{O}(U))\longrightarrow
  S_U := S\otimes_{\mathcal{O}(X_n)}\mathcal{O}(U),
\]
where $\mathrm{Tr}\circ\psi_U$ is the usual matrix trace. We show that $U$ satisfies the required properties.

A two‑sided ideal of $M_k(B)$ over a commutative ring $B$ is of the form $M_k(I)$ for a unique ideal $I\subset B$. If $\psi_U$ had nonzero kernel, then its intersection with some $M_{a_i}(\mathcal{O}(U))$ would be a nonzero two‑sided ideal $M_{a_i}(I)$ on which the trace vanishes. By trace‑compatibility of $\psi_U$, this forces $I=0$, a contradiction. Thus $\psi_U$ is injective, and we henceforth view it as an inclusion. For each $i$, let $e_i\in S_U$ denote the idempotent corresponding to the identity matrix in the $i$-th factor and zero elsewhere.

For $i\neq j$, set $I'_{i,j}=e_i S_U e_j S_U e_i\subset e_i S_U e_i$. Since $e_i S_U e_i = M_{a_i}(\mathcal{O}(U))$, each $I'_{i,j}$ is of the form $M_{a_i}(I_{i,j})$ for a unique ideal $I_{i,j}\subset\mathcal{O}(U)$, and $I_{i,j}=\mathrm{Tr}(I'_{i,j})$; the ideal $I_{i,j}$ is intrinsic, in the sense that it does not depend on the chosen GMA presentation. Define ideals
\[
  I:=\sum_{i\neq j} I_{i,j}, \qquad I_i:=\sum_{j\neq i} I_{i,j}.
\]
Let $U_{\underline{a}}:=\mathrm{Sp}(\mathcal{O}(U)/I)$ be the closed analytic subset of $U$ cut out by $I$. If $\alpha\in e_i S_U e_j$ with $i\neq j$, then $\alpha=e_i \beta e_j$ for some $\beta \in S_U$ and $\mathrm{Tr}(\alpha)=\mathrm{Tr}(e_i \beta e_j)=\mathrm{Tr}(e_j e_i \beta)=0$, since $e_j e_i=0$. Thus $\mathrm{Tr}(e_i S_U e_j)=0$ for $i\neq j$, and for
all $s\in S_U$,
\begin{equation}\label{eq:tr}
  \mathrm{Tr}(s)=\sum_{i=1}^r \mathrm{Tr}(e_i s e_i).
\end{equation}
For each $i$, the map
\[
  \rho_i : S_U \longrightarrow
  e_i S_U e_i / \sum_{j\neq i} e_i S_U e_j S_U e_i
  = M_{a_i}(\mathcal{O}(U)/I_{i})
  \longrightarrow M_{a_i}(\mathcal{O}(U)/I),
  \quad s\mapsto e_i s e_i,
\]
is an $\mathcal{O}(U)$\nobreakdash-algebra morphism, since for $s,s'\in S_U$,
\[
  e_i s e_i\cdot e_i s' e_i - e_i s s' e_i
    = e_i s (e_i-1) s' e_i
    = -\sum_{j\neq i} e_i s e_j s' e_i \in M_{a_i}(I_i).
\]

Let $T_i$ denote the pseudocharacter $T_i(g)=\mathrm{Tr}_i(\rho_i(g))\in\mathcal{O}(U)/I$, where $\mathrm{Tr}_i$ is the usual matrix trace on $M_{a_i}$. Then \eqref{eq:tr} implies $\mathrm{Tr}(g)\equiv \sum_{i=1}^r T_i(g) \mod{I}.$ Each $T_i$ is continuous since $T_i(g)=\mathrm{Tr}(e_i g)\bmod I$. Hence $T \vert_{U_{\underline{a}}}$ is a sum of $r$ $\mathcal{O}(U_{\underline{a}})$-valued pseudocharacters of dimensions $a_1, \ldots, a_r$.

Conversely, let $y \in U$ be an $\underline{a}$-reducible point with specialisation $T^0 : G \to k(y)$ equal to $T^0=T_1^0 + \cdots + T_r^0$, where (after a finite extension of scalars as for $x$), each $T_i^0:G \to k(y)$ is a continuous pseudocharacter of dimensions $a_i$. Set $S_y=S_U\otimes_{\mathcal{O}(U)}k(y)$. To show that $y \in U_{\underline{a}}$, it suffices to prove that the image of $I$ in $k(y)$ is zero.

By \cite[\S1.2.4]{BC09}, the characteristic polynomial identity $P_{x,T^0}=\prod_i P_{x,T_i^0}$ implies that each $T_i^0$ factors through a pseudocharacter of $S_y$. Via the inclusion $\psi_y:=\psi_U\otimes k(y)$, this further factors through the $i$‑th block $M_{a_i}(k(y))$. By \cite[Prop.~1.2.2]{BC09}, pseudocharacters of $M_{a_i}(B)$ over a commutative $\Q$‑algebra $B$ are integer multiples of the matrix trace, so there exist $m_{i, j} \in \Z_{\ge 0}$ such that $T_j^0(e_i)=m_{i,j} a_i$. Since
\begin{equation*}
    \sum_{j=1}^r m_{i, j}a_i =  \sum_{j=1}^r T_j^0(e_i)=T^0(e_i)=a_i,
\end{equation*}
each row of $(m_{i,j})$ contains precisely one $1$. Similarly, since $T_j^0$ has dimension $a_j$,
\[
  a_j = T_j^0(1)=  T_j^0 \left( \sum_{i=1}^r e_i \right) = \sum_{i=1}^r T_j^0(e_i)=  \sum_{i=1}^r m_{i,j} a_i
\]
forces each column to contain exactly one $1$, and the nonzero entry in column $j$ lies in a row $i$ with $a_i=a_j$. Thus $(m_{i,j})$ is a permutation matrix, and after reordering we may assume that $T_i^0(e_i)=a_i$ and $T_i^0(e_j)=0$ for $j\neq i$.

For $s\in S_y$, we can write $s=\sum_{k,\ell} e_k s e_\ell$. By construction, $T_i^0$ vanishes on all components except the $i$‑th diagonal block, so $T_i^0(e_k s e_\ell)=0$ unless $k=\ell=i$. On the $i$‑th block, $T_i^0$ is the matrix trace, so $T_i^0(e_i s e_i)=\mathrm{Tr}(e_i s e_i)$. Since $e_i$ is an idempotent, $e_i s e_i=e_i s$, and therefore
\[
  T_i^0(s)=T_i^0(e_i s e_i)=\mathrm{Tr}(e_i s e_i)=\mathrm{Tr}(e_i s).
\]
In particular, for $j\neq i$ and $s,s'\in S_y$, since $e_j e_i=0$,
\[
  T_i^0(e_i s e_j s' e_i)=\mathrm{Tr}(e_i s e_j s' e_i)
  =\mathrm{Tr}(e_j s' e_i s e_i)=0.
\]
Thus $T_i^0$ vanishes on $e_i S_y e_j S_y e_i$. Since $e_i S_y e_j S_y e_i = M_{a_i}(I_{i,j}^0)$, where $I_{i,j}^0$ is the image of $I_{i,j}$ in $k(y)$, we obtain $I_{i,j}^0=0$ for all $i\neq j$. Hence the image of $I=\sum_{i\neq j} I_{i,j}$ in $k(y)$ is zero. Since $y\in U$ is a $\underline{a}$‑reducible, $\rho_y$ admits an injective morphism
\[
  \psi_y : M_{a_1}(\overline{k(y)})\times\cdots\times
  M_{a_r}(\overline{k(y)}) \longrightarrow
  S_U \otimes_{\mathcal{O}(U)} \overline{k(y)}. 
\]
such that $\mathrm{tr}(\rho_y)\circ\psi_y$ is the usual matrix trace. As in the argument for $\underline{u}$ (now over $\overline{k(y)}$), the only possible decomposition of $\mathrm{tr}(\rho_y)$ as a sum of $r$  pseudocharacters is into $r$ pseudocharacters of dimensions $a_i$. Each summand is irreducible, by the Jacobson density theorem: a representation $G\to\GL_{a_i}(\overline{k(y)})$ is irreducible if and only if the induced map $\overline{k(y)}[G]\longrightarrow M_{a_i}(\overline{k(y)})$ is surjective. In our setting, the block $e_i\bigl(S_U \otimes_{\mathcal{O}(U)} \overline{k(y)}\bigr)e_i$ identifies with $M_{a_i}(\overline{k(y)})$, and the image of $\overline{k(y)}[G]$ contains this block because $\psi_y$ is injective and trace-compatible. Hence each summand is irreducible. Thus every $\underline{a}$‑reducible point in $U$ lies in $U_{\underline{a}}$ and is precisely $\underline{a}$‑reducible in the sense of Definition~\ref{def:a-red}, as required.
\end{proof}

\subsection{Pseudocharacters for $p$-adic families}

By the universal property of $X_4$, the pseudocharacter $T : G_{\Q,S} \to \mathcal{O}(X)$ attached to a $p$-adic family $X(\phi)$ is a pullback of the universal pseudocharacter on $X_4$ along $X \to X_4$. This allows us to study reducibility of the generic pseudocharacter of $X(\phi)$ using the geometric structure of $X_4$.

Let $K=\mathrm{Frac}(\mathcal{O}(X))$ and let $T_\eta := T \otimes_{\mathcal{O}(X)} K : G_{\Q,S} \to K$ denote the \emph{generic pseudocharacter} of a $p$-adic family $X(\phi)$. By Taylor's theorem \cite[Thm.~1.2]{Tay91}, there exists a unique (up to isomorphism) semisimple representation $\rho^{\mathrm{gen}} : G_{\Q,S} \to \GL_4(\overline{K})$ whose trace is $T_\eta \otimes_K \overline{K}$. A priori, the precise reducibility types of $T_\eta$ in the sense of Definition~\ref{def:a-red} are $\mathcal{S} :=\{(4), (3,1), (2,2), (2,1,1), (1,1,1,1)\}$. 

To determine which of these types can occur, we use the geometry of the universal pseudocharacter space $X_4$. The morphism $f : X \to X_4$ induced by the universal property of $X_4$ pulls back the analytic loci in $X_4$ cut out by precise $\underline{a}$-reducibility. Passing between reducibility statements over different base fields requires showing that the irreducible constituents of a semisimple representation remain irreducible after extending scalars, which is precisely the content of the following lemma.

\begin{lemma}\label{lem:JD}
Let $F\subset L$ be fields, and choose algebraic closures so that $\overline F\subset \overline L$. Let $V$ be a finite-dimensional semisimple representation of $G$ over $\overline F$. Then $V\otimes_{\overline F}\overline L$ is irreducible as a $\overline L[G]$-module if and only if $V$ is irreducible as a $\overline F[G]$-module.
\end{lemma}
\begin{proof}
    This follows from the Jacobson density theorem.
\end{proof}

\begin{prop} \label{prop: Zar dense from 1.1}
Let $X(\phi)$ be a positive-dimensional $p$-adic family with pseudocharacter $T: G_{\Q, S} \to \mathcal{O}(X)$ and let $\underline{a}=(a_1, \ldots, a_r) \in \mathcal{S}$. Suppose that $X$ contains a subset $Z' \subseteq Z$ of precisely $\underline{a}$-reducible points that accumulates at precisely $\underline{a}$-reducible point $y \in Z$. Then the generic pseudocharacter $T_\eta: G_{\Q, S} \to \mathrm{Frac}(\mathcal{O}(X))$ is precisely $\underline{a}$-reducible.
\end{prop}

\begin{proof}
Let $f: X \to X_4$ be the morphism in the category of rigid spaces induced by the universal property of $X_4$, so that $T: G_{\Q, S} \to \mathcal{O}(X)$ is the pullback of the universal pseudocharacter $\mathrm{Tr}: G_{\Q, S} \to \mathcal{O}(X_4)$. Let $L/k(y)$ be a finite extension over which $\rho_{\underline{y}}$ is defined and let $y_L$ be the corresponding point lying over $y$.
By Lemma~\ref{lem: 1.1}, there exists an affinoid neighbourhood $U \subset X_{4,L}$ of $f_L(y_L)$ and a closed subset $U_{\underline{a}} \subset U$ consisting of the precisely $\underline{a}$-reducible points of $U$, where $f_L : X_L \to X_{4,L}$ denotes the base change of $f$. Since $Z'$ accumulates at $y$, the set $Z_L' := Z'\times_{\Q_p} L$ accumulates at $y_L$, and there exists an irreducible, reduced affinoid neighbourhood $W \subset X_L$ of $y_L$ such that $Z_L' \cap W$ is Zariski dense in $W$. Moreover $f_L^{-1}(U_{\underline{a}}) \cap W$ is closed in $W$, so taking closures inside $W$ gives
\[
W = \overline{Z_L' \cap W} \subset \overline{f_L^{-1}(U_{\underline{a}}) \cap W} = f_L^{-1}(U_{\underline{a}}) \cap W.
\]
Hence every point of $W$ is precisely $\underline{a}$-reducible. Thus, writing $T_{W}: G_{\Q, S} \to \mathcal{O}(W)$ and $K_W = \mathrm{Frac}(\mathcal{O}(W))$, the pseudocharacter $T_{K_W} := T_W \otimes_{\mathcal O(W)} K_W$ is precisely $\underline{a}$-reducible, so the unique (up to isomorphism) semisimple representation $\rho^{\mathrm{gen}}_W : G_{\Q,S} \to \GL_4(\overline{K_W})$ with trace $T_{K_W} \otimes_{K_W} \overline{K_W}$ (given by Taylor's theorem~\cite[Thm.~1.2]{Tay91}) is precisely $\underline{a} = (a_1, \ldots, a_r)$-reducible. Since $K := \mathrm{Frac}(\mathcal{O}(X)) \hookrightarrow K_W$, choosing algebraic closures $\overline{K}\subset \overline{K_W}$, we obtain $T_{K_W} \otimes_{K_W} \overline{K_W} = T_\eta \otimes_{K} \overline{K_W}$, and so $\rho^{\mathrm{gen}}_W \simeq \rho^{\mathrm{gen}}\otimes_{\overline K}\overline{K_W}$. Applying Lemma~\ref{lem:JD} to each irreducible constituent of $\rho^\mathrm{gen}$ with $F = K$ and $L =K_W$ shows that $T_\eta$ is precisely $\underline{a}$‑reducible.
\end{proof}
\noindent
It is also important to understand how reducibility behaves under specialisation. The next lemma shows that $\underline{a}$-reducibility of $T_\eta$ is inherited by $T_z$ for each $z\in Z$.

\begin{defn}
\label{def:res-mult-free}
Let $A$ be a henselian local ring with maximal ideal $\mathfrak{m}$ and residue field $k= A/\mathfrak{m}$. A pseudocharacter $T: R \to A$ is \emph{residually multiplicity-free} if $T \mod \mathfrak{m}$ is multiplicity-free. 
\end{defn}

\begin{rem}
\label{rem:res-mult-free-z}
Let $X(\phi)$ be a $p$-adic family as in Definition~\ref{def:p-adic-family-GSp4}, and let $z \in Z$. The residual Galois representation $\rho_z$ is defined over a finite extension $k_z/k(z)$, where $k(z)$ is the residue field of $\mathcal{O}(X)$ at $z$ (\cite[244]{Ski09}). By Definition~\ref{def:p-adic-family-GSp4}, $\rho_z$ is multiplicity-free, so $T_z = \ev_z\circ T$ is multiplicity-free. Let $X_{k_z} := X \times_{\Q_p} k_z$ and denote by $z' \in X_{k_z}(k_z)$ the corresponding point lying over $z$ and by $A_z := \mathcal{O}_{X_{k_z},z'}$ the local ring of $X_{k_z}$ at $z'$, with residue field $k_z$. Then $T_{A_z} := T \otimes_{\mathcal{O}(X)} A_z$ is residually multiplicity-free in the sense of Definition~\ref{def:res-mult-free}. 
\end{rem}

\begin{prop} \label{prop: T_eta to T_z}
Let $\underline{a} = (a_1, \ldots, a_r) \in \mathcal{S}$ and suppose that the generic pseudocharacter $T_\eta$ is $\underline{a}$-reducible. Then for all $z \in Z$, $T_z : G_{\Q,S} \to k_z$ is $\underline{a}$-reducible.
\end{prop}

\begin{proof}
Let $k_z$, $\rho_z$ and $z' \in X_{k_z}(k_z)$ be as in Remark~\ref{rem:res-mult-free-z} and let $Y$ be the irreducible component of $X_{k_z}$ containing $z'$. Let $A_z := \mathcal{O}_{Y,z'}$ be the local ring of $Y$ at $z'$ and denote by $K_z = \mathrm{Frac}(A_z)$ its fraction field, $\mathfrak m_z$ its maximal ideal, and $k_z$ its residue field. Set $T_{A_z} := T \otimes_{\mathcal{O}(X)} A_z$ and  $T_{\eta,z} := T_{A_z} \otimes_{A_z} K_z$. Since $K\hookrightarrow K_z$ as $\mathcal{O}(X)\hookrightarrow A_z$, choosing algebraic closures $\overline{K} \xhookrightarrow{} \overline{K_z}$, we have  $T_{\eta,z}\otimes_{K_z}\overline{K_z}\simeq T_\eta\otimes_K\overline{K_z}$. Hence $T_{\eta,z}$ is $\underline a$-reducible, since $T_\eta$ is. By Remark~\ref{rem:res-mult-free-z}, $T_{A_z}$ is residually multiplicity-free, so \cite[Prop.~1.5.1]{BC09} applies to $T_{A_z}$ with the partition $\mathcal{P}=(a_1,\ldots,a_r)$, and gives an ideal $\mathcal{I}_{\mathcal{P},z}\subset A_z$ such that for every ideal $J\subset A_z$,  $T_{A_z}\otimes_{A_z} A_z/J$ is $\underline a$-reducible if and only if $\mathcal{I}_{\mathcal{P},z}\subset J$. Since $T_{\eta,z}$ is $\underline{a}$-reducible, the image of $\mathcal{I}_{\mathcal{P},z}$ in $K_z$ is zero, hence $\mathcal{I}_{\mathcal{P},z}=0$ in $A_z$ by injectivity of $A_z \to K_z$. Taking $J = \mathfrak m_z$ shows that $T_z = T_{A_z} \otimes_{A_z} k_z$ is $\underline{a}$-reducible, as required. 
\end{proof}

\subsection{Non-$(2,1,1)$-reducibility of $T_\eta$}
We now show that the generic pseudocharacter of a $p$-adic family through an SK point cannot be $(2,1,1)$-reducible. The argument combines the extension and specialisation of reducibility (Propositions~\ref{prop: Zar dense from 1.1} and \ref{prop: T_eta to T_z}), Saito--Kurokawa rigidity (\S \ref{s: Rigidity for families of SK points}), and the structure of SK points and Arthur’s classification (\S\ref{s: Arthur's Classification}).
\begin{defn} \label{def: points}
Let $X(\phi)$ be a $p$-adic family. A point $z=(\Pi,\psi_\mathcal{R})\in Z$ is a \emph{type $(x)$ point} if $\Pi$ is of type~$(x)$ in the sense of Arthur's classification.
\end{defn}

\begin{prop} \label{prop: T-not-(2,1,1)}
Let $X(1)$ be a positive-dimensional $p$-adic family through an SK point $z_0$ as in Proposition \ref{prop:ref 2/3}. Then $T_\eta$ is not $(2,1,1)$-reducible.
\end{prop}
\begin{proof}
Suppose that $T_\eta$ is $(2,1,1)$-reducible. Then by Proposition~\ref{prop: T_eta to T_z}, $T_z$ is $(2,1,1)$-reducible for all $z\in Z$. Since type $(a)$ points are irreducible \cite[Thm.~1.1]{Wei22} and type $(b)$/$(c)$ points are precisely $(2,2)$-reducible by the irreducibility of Galois representations associated to cuspidal modular forms, Arthur's classification implies that $Z=Z_{ef} \sqcup Z_\mathrm{SK}$, where $Z_{ef}$ (resp.\ $Z_{\mathrm{SK}}$) denotes the subset of $Z$ consisting of type $(e)/(f)$ (resp.\ SK) points, cf. \S\ref{s: Arthur's Classification}.

Since $Z$ accumulates at $z_0$ by definition and $Z_{\mathrm{SK}}$ does not by Proposition~\ref{prop:ref 2/3}, there exists an affinoid neighbourhood $V$ of $z_0$ such that $Z\cap V$ is dense in $V$ but $Z_{\mathrm{SK}}\cap V$ is not dense in $V$. It follows that $Z_{ef}\cap V = (Z\cap V)\setminus (Z_{\mathrm{SK}}\cap V)$ is dense in $V$, and hence nonempty.

Choose $z_1\in Z_{ef}\cap V$. Since $Z_{\mathrm{SK}}\cap V$ is not dense, there exists a sufficiently small affinoid neighbourhood $W\subseteq V$ of $z_1$ such that $W\cap Z_{\mathrm{SK}}=\emptyset$, and hence $Z\cap W$ is dense in $W$ and consists entirely of type $(e)/(f)$ points. Thus $Z_{ef}$ accumulates at $z_1$.

By Arthur's classification, type $(e)/(f)$ points correspond to sums of four characters, so are $(1,1,1,1)$-reducible. Thus $Z' := Z_{ef}$ is a subset of precisely $(1,1,1,1)$-reducible points that accumulates at $z_1$. Applying Proposition~\ref{prop: Zar dense from 1.1} with $y = z_1$ and $\underline{a} = (1,1,1,1)$ implies that $T_\eta$ is $(1,1,1,1)$-reducible. Proposition~\ref{prop: T_eta to T_z} implies that $T_{z_0}$ is also $(1,1,1,1)$-reducible, contradicting the precise $(2,1,1)$-reducibility
of $z_0$. Thus $T_\eta$ is not $(2,1,1)$-reducible. 
\end{proof}

\subsection{The Kisin property}
To prove $p$-adic rigidity for $z_0$ we require that the extensions arising from $T$ are crystalline at $p$. This is ensured by the \emph{Kisin property}, an interpolation property for crystalline Frobenius eigenvalues.
\begin{defn}
\label{def:realisation}
A \emph{realisation} of a pseudocharacter $T\colon G_{\Q,S}\to \mathcal{O}(X)$ over $\mathcal{O}_z$ is a torsion-free finite $\mathcal{O}_z$-module $M$ with a continuous action of $G_{\Q,S}$ (with respect to the inverse limit topology on $G_{\Q,S}$, cf. Remark~\ref{rem: topology}) such that for all $g\in G_{\Q,S}$, the trace $\mathrm{tr}\bigl(g\mid M\otimes_{\mathcal{O}_z}\mathrm{Frac}(\mathcal{O}_z)\bigr)$ lies in $\mathcal{O}_z$ and equals the germ of $T(g)$ at $z$.
\end{defn}

\begin{rem}\label{rem: topology}
If $A$ is the local ring of a rigid analytic space at a closed point, then there is a natural way to equip any finite type $A$-module $M$ with a Hausdorff $A$-module topology, cf.\ \cite[\S1.5.5]{BC09}; this is the topology with which we equip $M$. The action of $G_{\Q,S}$ on $M$ is continuous when the map $G_{\Q,S}\times M \to M$ is continuous for the product topology. 
\end{rem}

\begin{defn}
\label{def:KP}
A $p$-adic family $X(\phi)$ satisfies the \emph{Kisin property} at a point $z\in X$ with residue field $k(z)$ if for any realisation $M$ of $T$ over $\mathcal{O}_z$, writing $M_z = M \otimes k(z)$, the following holds: 

\smallskip
\noindent
If $\dim \Dcris\left(M_z^{\mathrm{ss}}(\kappa_1(z))\right)^{\varphi=F_1(z)} \le 1$, then $\dim \Dcris\left(M_z(\kappa_1(z))\right)^{\varphi=F_1(z)} = 1.$
\end{defn}
\begin{rem} \label{rem:KP-def}
Our use of torsion-free realisations and the formulation of the Kisin property follow the philosophy of Bella\"iche \cite[Rem.~6]{Bel10}.
\end{rem}

\section{Complete rigidity}
\label{s:tot-rig}
We now describe the hypotheses on the SK point $z_0=(\pi,\psi_i)$ under which we establish complete rigidity. Assume that $z_0$ is as in Proposition~\ref{prop:ref 2/3}, and that $\pi$ is noncuspidal (equivalently, $S_{\mathrm{Sch}}=\emptyset$ and $L(1/2,\mu')\neq 0$). In particular, the choice of $S_{\mathrm{Sch}}$ implies that $z_0$ is nongeneric, in the sense that $\pi_{\ell}$ is nongeneric for every $\ell\neq p$ (cf. Remark~\ref{rem:genericity}). Assume moreover that $z_0$ satisfies
\begin{enumerate}[leftmargin = 2em]
    \item[(St)] for every prime $\ell\neq p$ such that $\mu_{\ell}$ is an unramified twist of Steinberg, $\mu_{\ell}$ is the non-trivial quadratic twist. In this case, the Atkin--Lehner eigenvalue $w_\ell$\footnote{As a newform of trivial nebentypus, $f \in S_{2k-2}(\Gamma_0(N))$ is an eigenvector for the Atkin--Lehner involution $W_\ell$ for each finite prime $\ell \mid N$: $W_\ell f = w_\ell(f) f$ with $w_{\ell} = w_\ell(f) \in \{\pm 1\}$. If $k=2$ and the $q$-expansion of $f$ has rational coefficients, then the corresponding elliptic curve has multiplicative reduction at $\ell$, nonsplit precisely when $w_\ell = 1$.} is $1$ (cf. Lemma~\ref{lem:no-epsd}).
\end{enumerate}

\begin{lemma} \label{lem:Selmer-vanishing}
$H_f^1(\Q,\rho_\mu(2))=H_f^1(\Q, \overline{\Q}_p(1))=0$.
\end{lemma}

\begin{proof}
By Lemma~\ref{lem:SK-cohomological-structure}, $\mu$ corresponds to a newform $f\in S_{2k-2}(\Gamma_0(N))$ of trivial nebentypus, where $p\nmid N=\mathrm{cond}(\mu)$, and $\rho_\mu = \rho_f(k-3)$. Since $L(1/2,\mu)=L(1/2,\mu') \neq 0$ and $L(s-k+3/2, \mu)=L(s,f)$ by \cite[\S 1.3]{SU06}, we have $L(k-1,f)=L(1/2,\mu)\ne0$. Applying Kato's theorem \cite[Theorem~14.2(2)]{Kat04} with $r=k-1$ and $\chi = 1$ shows that $H^1_f(\Q,\rho_\mu(2))=H^1_f(\Q,\rho_f(k-1))=0$. The vanishing of $H_f^1(\Q, \overline{\Q}_p(1))$ follows by Kummer theory.
\end{proof}
\noindent
To ensure that the extensions arising from $T$ are unramified outside $p$, we impose a monodromy control condition \textup{(SK--P2)}.
\begin{defn}
    \label{def:SK-P2}
Let $S_0\subset S$ be the set of primes $\ell\neq p$ at which $\mu_\ell$ is special. For each $\ell\in S_0$, let $N_{\pi,\ell}$ denote the monodromy operator of $\mathrm{std} \circ\mathrm{rec}_{\mathrm{GT}}(\pi_{\ell})$, viewed in $M_4(\overline{\Q}_p)$ via $\iota_p^{-1}:\C \xrightarrow{\sim} \overline{\Q}_p$ (cf.\ \S{\ref{s: The non-archimedean Local Langlands correspondence}}).
A $p$-adic family $X(1)$ through $z_0$ satisfies condition \textup{(SK--P2)} if for all $z \in Z$:
\begin{enumerate}
    \item[(i)] for all $\ell\in S_0$, the monodromy operator $N_\ell(\rho_z)$ of $\WD(\rho_z|_{W_{\Q_\ell}})$ lies in the Zariski closure of the conjugacy class of $N_{\pi,\ell}$, i.e. $N_\ell(\rho_z) \in \overline{\mathrm{Ad}(\GL_4(\overline{\Q}_p))\cdot N_{\pi,\ell}}$;
    \item[(ii)] for all $\ell\in S\setminus(\{p\}\cup S_0)$, $N_\ell(\rho_z)$ is trivial.
\end{enumerate}
\end{defn}

\begin{rem}
\label{rem:SK-P2-viable}
Condition~\textup{(SK--P2)} is modelled on~\textup{(P2)}, together with its special case \textup{(P3)}, of \cite[Conjecture~6.8.1 (Rep$(m)$)]{BC09}. The continuity of monodromy in $p$-adic families, formalised for unitary groups in \cite[Prop.~7.8.19]{BC09}, shows that the generic Jordan type of the monodromy operator is constant along each irreducible component and that monodromy can only increase on Zariski closed subsets. In particular, a monodromy bound verified on a Zariski dense and accumulating set of classical points propagates to every point of the family, a key input in Proposition~\ref{prop:trivial-mon}.  

The $\GL_2{_{/\Q}}$ case illustrates the feasibility of $p$-adic families with trivial monodromy outside $p$: Knightly’s theorem for modular forms whose local factors are supercuspidal at every prime dividing the level \cite{Kni25} provide an infinite source of classical points with everywhere trivial monodromy. Since full local--global compatibility is known for all points $z \in X(1)$ that are cohomological and of general type by \cite[Thm.~3.1]{Mok14}, it should be possible to verify \textup{(SK--P2)} in practice.
\end{rem}

\begin{thm}
\label{thm: X is a point}
Let $z_0$ be a noncuspidal $\psi_i$-refined SK point satisfying \textup{(St)} with $i\in\{2,3\}$ and $v(\alpha(z_0))\neq k(z_0)-2$ when $i=3$. Let $X(1)$ be a $p$-adic family through $z_0$ satisfying condition \textup{(SK--P2)} and the Kisin property at $z_0$. Then $X(1)$ is a point.
\end{thm}

\noindent
We now prove Theorem~\ref{thm: X is a point}, following \cite[\S\S8.2–8.3]{BC09}. Let $A=\mathcal{O}_{z_0}$ be the rigid local ring of $X$ at $z_0$, with maximal ideal $\mathfrak{m}$, residue field $k = A/ \mathfrak{m}$, and fraction field $K=\mathrm{Frac}(A)$. The pseudocharacter $T\colon G_{\Q,S}\to\mathcal{O}(X)$ induces a continuous pseudocharacter $T\otimes_{\mathcal{O}(X)}A\colon G_{\Q,S}\to A$, which we again denote by $T$. It defines a two-sided ideal
\[
  \ker T = \{x\in A[G_{\Q,S}] \mid T(xy)=0\ \text{for all } y\}.
\]
Since $\rho_{z_0}=\rho_\mu\oplus\epsilon^{-2}\oplus\epsilon^{-1}$ is residually multiplicity-free, \cite[Thm.~1.4.4]{BC09} implies that the Cayley--Hamilton quotient $R:=A[G_{\Q,S}]/\ker T$ is a GMA over $A$, torsion-free and of finite type. In particular, we may choose idempotents $e_{\rho_\mu}, e_{\epsilon^{-2}}, e_{\epsilon^{-1}} \in R$ corresponding to the irreducible constituents of $\rho_{z_0}$, together with a representation $R\otimes_A K\to M_4(K)$ adapted to these idempotents, as in \cite[Thm.~1.4.4]{BC09}. 

Let $\rho_K\colon G_{\Q,S}\to\GL_4(K)$ be the induced representation; it has trace $T$ and kernel $\ker T$. Since $\rho_K$ is semisimple by \cite[Lem.~4.3.9(i)]{BC09} and a direct sum of absolutely irreducible $K$-representations, applying \cite[Lem.~4.3.7, Prop.~7.8.14]{BC09} to each irreducible constituent shows that for each $\ell \neq p$, $\rho_K \vert_{G_{\Q_\ell}}$ admits an associated Weil--Deligne representation $(r_\ell, N_\ell(\rho_K))$. Moreover, it follows from the proof of \cite[Prop.~7.8.14]{BC09} that $N_\ell(\rho_K)$ is in $R_\ell$, where $R_\ell \subset R$ is the image of $A[G_{\Q_\ell}]$ in $R$ via the map $A[G_{\Q, S}] \to R$. Applying Taylor’s theorem \cite[Thm.~1.2]{Tay91} to $T\colon G_{\Q,S}\to A$, there exists a a unique (up to isomorphism) semisimple representation $\rho_{z_0}^{\mathrm{gen}}\colon G_{\Q,S}\to\GL_4(\overline{K})$ with trace $T\otimes \overline{K}\colon G_{\Q,S}\to \overline{K}$. Since $\rho_{z_0}^{\mathrm{gen}}$ and $\rho_K$ have the same trace, semisimplicity and the Brauer--Nesbitt theorem imply that $\rho_{z_0}^{\mathrm{gen}}\simeq\rho_K\otimes_K\overline{K}$.

\begin{defn}
\label{def: Ext T for family}
Let $\mathcal{I}=\{\epsilon^{-2},\epsilon^{-1},\rho_\mu\}$ be the set of irreducible constituents of $\rho_{z_0}$.  For $i,j\in\mathcal{I}$, define
\[
  \Ext_T(i,j) := \Ext^1_{R\otimes_A k}(i,j) =\Ext^1_R(i,j),
\]
where the equality follows from \cite[Rem.~1.5.9]{BC09}.
\end{defn}  
\noindent
In particular, elements of $\Ext_T(i,j)$ may be viewed as extension classes inside the GMA $R$, and hence correspond to Galois extensions realised in subquotients of $R$-modules. Since $R$ is finite type over $A$, $R\otimes_A k$ is a finite-dimensional $k$-algebra. By \cite[Thm.~1.5.5]{BC09}, each $\Ext_T(i,j)$ is therefore a finite-dimensional $k$-vector space. By \cite[Lem.~8.2.7]{BC09}, the image of the natural $k$-linear injection
\begin{equation*}
    \Ext_T(i,j)\hookrightarrow \Ext^1_{k[G_{\Q,S}]}(i,j)
\end{equation*}
is contained in the subspace of continuous extensions of $i$ by $j$ as $k[G_{\Q,S}]$-representations.

The strategy of the argument is as follows. Any $U \in \Ext_T(\epsilon^{-2},\epsilon^{-1})$ or $\Ext_T(\epsilon^{-2},\rho_\mu)$ arises as a subquotient of a realisation of $T$ and is crystalline at $p$ and unramified outside $p$ (Propositions~\ref{prop: nontriv subquots} and~\ref{prop:trivial-mon}). The non-$(2,1,1)$-reducibility of $T_\eta$ (Proposition~\ref{prop: T-not-(2,1,1)}) forces the total reducibility ideal $\mathcal{I}_\mathcal{P}$, the reducibility ideal attached to the finest partition $\mathcal{P}$ of $\mathcal{I}$, to be nonzero, so one of these two extension groups must be nontrivial. Twisting by $\epsilon^{2}$ and applying Selmer vanishing results (Lemma~\ref{lem:Selmer-vanishing}) then yields a contradiction. 

\begin{lemma}
\label{lem:no-epsd}
Let $f \in S_{2k-2}(\Gamma_0(N))$ be a holomorphic newform with trivial nebentypus, where $2k-2 \in \Z_{\ge 2}$, and let $\mu$ be the trivial central character twist of the associated automorphic $\GL_2(\A_\Q)$-representation. Let $\ell\neq p$ and $V=\rho_\mu\vert_{G_{\Q_\ell}}$. If $d \in \Z$, then
\[
  \Ext^1_{L[G_{\Q_\ell}]}\bigl(\epsilon^d,V\bigr)
  \;=\;
  \Ext^1_{L[G_{\Q_\ell}]}\bigl(V,\epsilon^d\bigr)
  \;=\; 0
\]
unless $\mu_\ell$ is an unramified twist of Steinberg and $(w_\ell,d)$ lies in $\{(-1,-1),(-1,-3)\}$ or $\{(-1,-2),(1,0)\}$ respectively, where $w_\ell \in\{\pm1\}$ is the Atkin--Lehner eigenvalue of $f$ at $\ell$ and $L$ is the field of definition of $\rho_\mu$. 
\end{lemma}
\begin{proof}

By Carayol’s proof of the compatibility of $\rho_f$ with the local Langlands correspondence for $\GL_2$ \cite{Car86}, the inverse characteristic polynomial $P(V,X) :=  \det(1 - X \cdot \Frob_\ell \mid V^{I_\ell})$ is given by 
\[P(V,X) =
\begin{cases}
  (1 - \alpha_\ell \ell^{-(k-3)} X)(1 - \beta_\ell \ell^{-(k-3)} X),
  & \text{if }\mu_\ell\text{ is unramified}\\[4pt]
  1 + w_\ell \ell X , & \text{if }\mu_\ell\text{ is an unramified twist of Steinberg}\\[4pt]
  1, & \text{otherwise }
\end{cases}
\]
with $|\alpha_\ell|_\C = |\beta_\ell|_\C =\ell^{k-3/2}$ and $w_\ell =\pm 1$ the Atkin--Lehner eigenvalue of $f$ at $\ell$. Hence $\ell^{-d}$ is a Frobenius eigenvalue on $V^{I_\ell}$ only in the second case, and then precisely when $(w_\ell,d)=(-1,-1)$. Hence this is the only case for which $\epsilon^{d}$ can occur as a subrepresentation of $V$. From the computations above and the local Euler characteristic formula, if $W= V(-d)$ or $V^\star(d)$, then $H^1(G_{\Q_\ell},W)$ can only be nonzero if $\mu_\ell$ is an unramified twist of Steinberg, in which case $V$ fits into a non-split extension
\[
  0 \longrightarrow \mathrm{ur}(w_\ell)\,\epsilon^{-1} \longrightarrow V
  \longrightarrow \mathrm{ur}(w_\ell) \epsilon^{-2} \longrightarrow 0,
\]
where $\mathrm{ur}(w_\ell)$ is the unramified character sending $\Frob_\ell$ to $w_\ell\in\{\pm1\}$. The desired statement follows from the identifications $\Ext^1_{L[G_{\Q_\ell}]}\bigl(\epsilon^d,V\bigr)= H^1(G_{\Q_\ell},V(-d))$  and $  \Ext^1_{L[G_{\Q_\ell}]}\bigl(V,\epsilon^d\bigr)$ $= H^1(G_{\Q_\ell}, V^\star(d))$.
\end{proof}

\begin{lemma}[{\cite[Lem.~8.2.12]{BC09}}]
\label{lem:central-idempotent-Rl}
For each $\ell\neq p$ there is a datum of idempotents $\{e_{\rho_\mu},e_{\epsilon^{-2}},e_{\epsilon^{-1}}\}$ for the GMA $R$ such that $e_\ell \;:=\; e_{\epsilon^{-2}} + e_{\epsilon^{-1}}$ is in the centre of $R_\ell$, where $R_\ell\subset R$ is the image of $A[G_{\Q_\ell}]$ in $R$.
\end{lemma}
\begin{proof}
A simple modification of \cite[Lem.~8.2.12]{BC09}, noting that Lemma~\ref{lem:no-epsd} and assumption \textup{(St)} allows us to apply \cite[Lem.~8.2.11]{BC09}.
\end{proof}

\begin{prop}
\label{prop:trivial-mon}
For each prime $\ell\neq p$, and for $(i, j) \in \mathcal{I}'= \{(\epsilon^{-2},\epsilon^{-1}), (\epsilon^{-2},\rho_\mu)\}$ every extension in $\Ext_T(i, j)$ is split when restricted to $I_\ell$.
\end{prop}

\begin{proof} 
Fix a prime $\ell\neq p$ and let $U \in \Ext_T(i,j)$ for $(i, j) \in \mathcal{I}'$. Since representations of finite groups in characteristic zero are semisimple by Maschke's theorem, it suffices to show $I_\ell$ acts through a finite quotient on $U$. Moreover, by \cite[Thm.~1.5.5]{BC09}, if the image of $N_\ell(\rho_K)$ in $R \subset R \otimes K$ is trivial, then the action of $I_\ell$ on $V$ factors through a finite quotient, and hence $U|_{I_\ell}$ is split. 

\smallskip
\noindent\emph{Case 1: $\ell\notin S$.}
Since $U$ is a $G_{\Q,S}$-representation (as $T$ is a pseudocharacter of $G_{\Q,S}$), $I_\ell$ acts trivially on $U$, and hence $U|_{I_\ell}$ is split.

\smallskip
\noindent \emph{Case 2: $\ell\in S\setminus(S_0\cup\{p\})$.}
By \textup{(SK--P2)(ii)}, for all $z\in Z$ we have $N_\ell(\rho_z)=0$. Applying \cite[Prop.~7.8.19]{BC09} with $x=z_0$ and $W=X$ gives an affinoid open $\Omega\subset X$ of $z_0$ and a Zariski-open-and-dense subset $\Omega'\subset\Omega$ such that $N_\ell(\rho_y)\sim_{I_\ell}N_\ell(\rho_{z_0}^{\mathrm{gen}})$ for all $y\in\Omega'$. Since $Z$ is Zariski dense in $X$, $Z\cap\Omega'$ is nonempty, so there exists $z\in Z\cap\Omega'$ such that $0= N_\ell(\rho_z)\sim_{I_\ell}N_\ell(\rho_{z_0}^{\mathrm{gen}})$, which implies $N_\ell(\rho_{z_0}^{\mathrm{gen}})=0$. As $\rho_{z_0}^{\mathrm{gen}}\simeq\rho_K\otimes_K\overline{K}$, their associated Weil--Deligne representations have conjugate monodromy operators, so the image of $N_\ell(\rho_K)$ in $R \subset R\otimes_A K$ is trivial, and hence $U|_{I_\ell}$ is split. 

\smallskip
\noindent \emph{Case 3: $\ell\in S_0$.}
If $(i,j)=(\epsilon^{-2},\rho_\mu)$, then by Lemma~\ref{lem:no-epsd}, since $w_\ell \neq 1$ by \textup{(St)}, we have $\Ext^1_{k[G_{\Q_\ell}]}\bigl(\epsilon^{-2},\rho_\mu\vert_{G_{\Q_\ell}}\bigr)=0$, and so the restriction of $U$ to $G_{\Q_\ell}$, and hence to $I_\ell$, is split. Let $(i,j)=(\epsilon^{-2},\epsilon^{-1})$. Choose a datum of idempotents as in Lemma~\ref{lem:central-idempotent-Rl} and a representation $R\otimes_A K\to M_4(K)$ adapted to those idempotents. By \cite[Thm.~1.5.5, Thm.~1.5.6(1)]{BC09} it suffices to show that the image $e_\ell N_\ell(\rho_K)$ of $N_\ell(\rho_K)$ in $e_\ell R_\ell e_\ell = e_\ell R_\ell  \subset e_\ell R_K e_\ell$ is trivial. By \textup{(SK--P2)(i)}, for every $z\in Z$ the monodromy $N_\ell(\rho_z)$ lies in the Zariski closure of the conjugacy class of $N_{\pi,\ell}$, so by \cite[Prop.~7.5.8]{BC09}, there exists $z \in Z$ such that
$N_\ell(\rho_{z_0}) \prec N_\ell(\rho_{z_0}^\mathrm{gen})  \sim N_\ell(\rho_K) \sim N_\ell(\rho_{z}^\mathrm{gen}) \sim N_\ell(\rho_{z}) \prec N_{\pi,\ell}$. Writing $(1-e_\ell)N_\ell(\rho_{z_0})$ for the image of $N_\ell(\rho_{z_0})$ in $\End(\rho_\mu)$, we have $N_{\pi, \ell} \sim N_\ell(\rho_{z_0})\sim (1-e_\ell) N_\ell(\rho_{z_0})$. But $(1-e_\ell)N_\ell(\rho_{z_0}) \prec (1-e_\ell)N_\ell(\rho_K)$ by \cite[Prop.~7.8.8]{BC09} applied to $(1-e_\ell)R_\ell (1-e_\ell)$ and $n=(1-e_\ell)N_\ell(\rho_K)$, so we obtain
\begin{equation*} \tag{$\star$}
    N_\ell(\rho_K) \prec N_{\pi, \ell} \sim (1-e_\ell)N_\ell(\rho_{z_0}) \prec (1-e_\ell)N_\ell(\rho_K).
\end{equation*}
By \cite[Prop 7.8.1]{BC09}, if $k$ is a field and $N, N' \in M_4(k)$ are nilpotent matrices, then $N \prec N'$ if and only if $\mathrm{rank}(N^i) \leq \mathrm{rank}((N')^i)$ for all $i \geq 1$. Since $N_{\pi,\ell}$ has rank one, $(1-e_\ell)N_\ell(\rho_K)$ has rank at most one, and so $(1-e_\ell)N_\ell(\rho_K) \prec N_\ell(\rho_K)$. Together with $(\star)$ this gives $N_\ell(\rho_K) \sim (1-e_\ell)N_\ell(\rho_K)$, and so $e_\ell N_\ell(\rho_K)=0$.
\end{proof}

\begin{prop}
\label{prop: nontriv subquots}
Every extension in $\Ext_T(\epsilon^{-2},\rho)$, where $\rho\in\{\rho_\mu,\epsilon^{-1}\}$, is crystalline at $p$.
\end{prop}

\begin{proof}
Let $G_p$ be a decomposition group at $p$ and write $\varphi$ for the crystalline Frobenius acting on $\Dcris(-)$. Let $U\in\Ext_T(\epsilon^{-2},\rho)$. By \cite[Thm.~1.5.6]{BC09}, $U$ is a quotient of $M_j/ \mathfrak{m} M_j\oplus\rho$ by a submodule $W$, where  $M_j = S E_j$ is the column $S$-submodule defined in \cite[\S1.5.4]{BC09} (and is a left ideal of $R \subset M_4(K)$) with $\rho_j = \epsilon^{-2}$, and every simple subquotient of $W$ is isomorphic to $\rho_\mu$ or $\epsilon^{-1}$. To show that $U$ is crystalline at $p$, it suffices to show that $\Dcris(U|_{G_p})^{\varphi=p^2}\neq 0$: by left exactness of $\Dcris(-)^{\varphi=*}$ and since $\rho\subset U$ is crystalline with $\varphi$-eigenvalue(s) distinct from $p$ and $p^2$ (for $\rho_\mu$ this follows since $\mu_p$ is an irreducible unramified principal series representation and $\det(\rho_\mu) = \epsilon^{-3}$), in this case
\[
\dim(\Dcris(U|_{G_p})) = \dim(\Dcris(\rho|_{G_p})+\dim(\Dcris(U|_{G_p}))^{\varphi=p^2}.
\]
Since $p^2$ does not occur as a $\varphi$-eigenvalue of $\epsilon^{-1}$ or $\rho_\mu$, it cannot occur as a $\varphi$-eigenvalue in the kernel of the surjection
\[
  (M_j/\mathfrak{m} M_j\oplus\rho)\longrightarrow
  U=(M_j/\mathfrak{m} M_j\oplus\rho)/W.
\]
Thus any $\varphi=p^2$ eigenspace in $\Dcris(U|_{G_p})$ must come from the $M_j/\mathfrak{m} M_j$-summand, and it suffices to show that $\Dcris((M_j/\mathfrak{m} M_j)|_{G_p})^{\varphi=p^2}\neq 0$. By \cite[Lem.~4.3.9]{BC09}, there exists an $R$-module $N\subset K^4$ such that $M:=N\oplus M_j$ has $K$-span $K^4$ and is a realisation of $T$, and $(N\otimes_A k)^{\mathrm{ss}}$ is a sum of copies of $\rho_\mu$ and $\epsilon^{-1}$. Since $z_0$ is $\psi_2$- or $\psi_3$-refined, we have $F_1(z_0)p^{\kappa_1(z_0)}=p^2$. Moreover, $(M_{z_0})^{\mathrm{ss}}\simeq \rho_{z_0}=\epsilon^{-1}\oplus\epsilon^{-2}\oplus\rho_\mu$, so
\[
  \dim\Dcris\bigl((M_{z_0})^{\mathrm{ss}}\vert_{G_p}\bigr)^{\varphi=p^2}=1.
\]
Hence $\dim \Dcris\!\left(M_{z_0}|_{G_p}\right)^{\varphi=p^2}=1$, by the Kisin property at $z_0$. Since the $\varphi=p^2$ eigenspace of $M_{z_0}|_{G_p}$ cannot lie in the kernel of the surjection $M_{z_0}\twoheadrightarrow M_j/\mathfrak{m} M_j$, it survives in the quotient. Hence $\Dcris((M_j/\mathfrak{m} M_j)|_{G_p})^{\varphi=p^2}\neq 0$, as required.
\end{proof}

\begin{lemma}
\label{lem: ext vanishes}
$\Ext_T(\epsilon^{-2},\rho_\mu)=0$.
\end{lemma}
\begin{proof}
Suppose that $U \in \Ext_T(\epsilon^{-2},\rho_\mu)$ is nontrivial. By Propositions~\ref{prop:trivial-mon} and \ref{prop: nontriv subquots}, $U$ is unramified outside $p$ and crystalline at $p$. The twist $U(2)$ of $U$ by the unramified/crystalline character $\epsilon^2$ defines a nontrivial class in $H_f^1(\Q,\rho_\mu(2))$, but this Selmer group vanishes by Lemma~\ref{lem:Selmer-vanishing}.
\end{proof}

\begin{lemma}
\label{lem: ext nonzero}
If $X(1)$ is not a point, then $\Ext_T(\epsilon^{-2},\epsilon^{-1})\neq 0$.
\end{lemma}

\begin{proof}
Suppose that $\Ext_T(\epsilon^{-2},\epsilon^{-1})=0$. Twisting by $\epsilon^2$, $\Ext_{T(2)}(1,\epsilon)=\Ext_{T(2)}(1,\rho_\mu(2))=0$ by assumption and Lemma~\ref{lem: ext vanishes}. By Lemma~\ref{lem: pseudo involution} and Remark~\ref{rem: tau-basechange}, the induced $\tau$-involution on the constituents $\{ \rho_\mu(2), 1, \epsilon \}$ of $T(2)$ satisfies the hypotheses of \cite[\S8.2.6]{BC09}. Applying \cite[Lem.~8.3.2]{BC09} to $T(2)$ with $j=1$, $\rho_1=\rho_\mu(2)$ and $\mathcal{P}=\{\{1\},\{2\},\{3\}\}$, the two vanishing conditions imply that the total reducibility ideal $\mathcal{I}_{\mathcal{P}}(T)$ is zero. By \cite[Prop.~1.5.1]{BC09}, this is equivalent to $T$ decomposing as a sum of pseudocharacters $T_1, T_2, T_3$ of respective dimensions $2,1,1$, and therefore $T_{\eta,z_0}:G_{\Q,S}\to \mathrm{Frac} (\mathcal{O}_{z_0})$ is $(2,1,1)$-reducible. If $X(1)$ is positive-dimensional, this contradicts Proposition~\ref{prop: T-not-(2,1,1)} (cf. the proof of Proposition~\ref{prop: T_eta to T_z}).
\end{proof}

\begin{proof}[Proof of Theorem~\ref{thm: X is a point}]
If $X(1)$ is not a point, then by Lemma~\ref{lem: ext nonzero} there exists a nontrivial $U\in\Ext_T(\epsilon^{-2},\epsilon^{-1})$. By Propositions~\ref{prop:trivial-mon} and~\ref{prop: nontriv subquots}, $U$ is unramified outside $p$ and crystalline at $p$. Thus the twist $U(2)$ of $U$ by the unramified/crystalline character $\epsilon^2$ defines a nontrivial class in $H_f^1(\Q, \overline{\Q}_p(1))$, but this Selmer group vanishes by Lemma~\ref{lem:Selmer-vanishing}. Therefore $X(1)$ is a point.
\end{proof}

\begin{rem}
\label{rem:genericity}
The noncuspidality of $z_0$ is not, a priori, a necessary condition for complete rigidity. Schmidt's construction \cite[Thm.~3.1]{Sch05} produces cuspidal $\GSp_4(\A_\Q)$ Saito--Kurokawa representations with $\epsilon(1/2,\mu)=1$ by taking $S_{\mathrm{Sch}} \not \ni p$ of nonzero even cardinality. Hence one may choose $\mu$ so that $L(1/2,\mu)\neq 0$, for which Kato's theorem still applies. However, for each $\ell\in S_{\mathrm{Sch}} \neq \emptyset$, the local factor $\pi_\ell$ is \emph{generic}; in particular $\pi_\ell$ is generic for some finite prime $\ell \neq p$, and $N_{\pi, \ell} \not \sim N_\ell(\rho_{z_0})$ (cf.~the classification in of SK (type~\textbf{(P)}) local factors  \cite[Table~2]{Sch20}). Replacing $N_{\pi, \ell}$ by $N_\ell(\rho_{z_0})$ in \textup{(SK--P2)}, the same method shows that the corresponding  SK point $z_0$ does not deform in a $p$-adic family as in Theorem~\ref{thm: X is a point}, i.e. the $p$-adic families (that we expect to exist) don't satisfy \textup{(SK--P2)} for the smaller Galois monodromy operator. Instead, we expect the monodromy will be of type $N_{\pi, \ell}$. 
\end{rem}

\begin{rem} \label{rem:T-irr-lattices-short}
If $T:G_{\Q,S} \to \mathcal{O}_{z_0}$ is residually multiplicity-free and $T_{\eta, z_0}: G_{\Q, S} \to \mathrm{Frac} (\mathcal{O}_{z_0})$ is irreducible (so that $\rho_K$ is absolutely irreducible, where $K=\mathrm{Frac}(\mathcal{O}_{z_0})$), then the argument above admits a reformulation in terms of lattices and Ribet's lemma as in \cite{BC04}. 
\end{rem}

\printbibliography
\end{document}